\documentclass[12pt]{article}
\usepackage{hyperref,amsmath,amssymb,amscd,amsthm,makeidx,txfonts,graphicx,url,bm,relsize}
\usepackage{a4wide,xy}
\usepackage{float}
\usepackage{tikz}

\let\oldmarginpar\marginpar
\renewcommand\marginpar[1]{\-\oldmarginpar[\raggedleft\footnotesize #1]
{\raggedright\footnotesize #1}}

\makeindex

\xyoption{all}
\input{xypic}

\numberwithin{equation}{section}

\newtheorem{theorem}{Theorem}[subsection]
\newtheorem{proposition}[theorem]{Proposition}
\newtheorem{corollary}[theorem]{Corollary}
\newtheorem{conjecture}[theorem]{Conjecture}
\newtheorem{lemma}[theorem]{Lemma}

\theoremstyle{remark}
\newtheorem{remark}[theorem]{Remark}
\newtheorem{example}[theorem]{Example}

\theoremstyle{definition}

\def\bes{\begin{eqnarray*}}
\def\ees{\end{eqnarray*}}

\def\bm{{\bm a}}

\def\F{{\rm F}}

\def\bfX{{\bf X}}
\def\pr{{\rm pr}}
\def\calM{\mathcal{M}}
\def\dcb{\mathcal{D}_c^b}
\def\IC{{\rm IC}}

\DeclareMathOperator{\Spec}{Spec} 
\DeclareMathOperator{\Tr}{Tr} \DeclareMathOperator{\Ind}{Ind}
\DeclareMathOperator*{\bigboxtimes}{\mathlarger{\mathlarger{\mathlarger\boxtimes}}}
\def\calI{{\mathcal{I}}}
\def\calX{{\mathcal{X}}}
\def\calL{{\mathcal{L}}}
\def\calY{{\mathcal{Y}}}
\def\car{{\mathfrak{car}}}

\def\calF{{\mathcal{F}}}

\def\calH{\mathcal{H}}
\def\calT{{\mathcal{T}}}

\def\frakg{{\mathfrak{g}}}
\def\fraku{{\mathfrak{u}}}
\def\frakt{{\mathfrak{t}}}
\def\frakb{{\mathfrak{b}}}
\def\frakl{{\mathfrak{l}}}

\def\1{{\bf 1}}

\def\I{{\rm I}}

\def\F{\mathbb{F}}

\def\Q{\overline{\mathbb{Q}}_\ell}

\def\T{{\rm{T}}}

\def\t{{\rm t}}
\def\calC{{\mathcal C}}

\def\Z{\mathbb{Z}}

\def\gl{{\rm gl}}

\newcommand{\nc}{\newcommand}
\nc{\bM}{{\mathbb M}}
\nc{\Sp}{{\rm Sp}}

\nc{\R}{{\rm R}}

\def\Nr{{\rm Nr}}

\def\calS{{\mathcal{S}}}

\nc{\Rt}{{\rm R}_2}

\nc{\op}[1]{\mathop{\mathchoice{\mbox{\rm #1}}{\mbox{\rm #1}}
{\mbox{\rm \scriptsize #1}}{\mbox{\rm \tiny #1}}}\nolimits}
\nc{\al}{\alpha}

\nc{\ep}{\varepsilon} \nc{\ga}{\gamma} \nc{\Ga}{Q}
\nc{\la}{\lambda} \nc{\La}{\Lambda} \nc{\si}{\sigma}
\nc{\Sig}{{Q}} \nc{\Om}{\Omega} \nc{\om}{\omega}
\nc{\coker}{{\rm coker}}

\nc{\SL}{{\rm SL}} \nc{\GL}{{\rm GL}} \nc{\PGL}{{\rm PGL}}
\nc{\G}{{\rm G}}
\nc{\bV}{{\mathbb V}}
\nc{\card}{{\rm card}}

\nc{\beq}[1]{\begin{eqnarray}\label{#1}}
\nc{\eeq}{\end{eqnarray}}

\def\U{{\rm U}}

\nc{\cpt}{{\op{cpt}}} \nc{\Dol}{{\op{Dol}}} \nc{\DR}{{\op{DR}}}
\nc{\B}{{\op{B}}} \nc{\Triv}{\op{Triv}} \nc{\Hod}{{\op{Hod}}}
\nc{\Log}{{\op{Log}}} \nc{\Exp}{{\op{Exp}}} \nc{\Est}{E_{\op{st}}}
\nc{\Hst}{H_{\op{st}}} \nc{\Left}[1]{\hbox{$\left#1\vbox to
  10.5pt{}\right.\nulldelimiterspace=0pt \mathsurround=0pt$}}
\nc{\kight}[1]{\hbox{$\left.\vbox to
  10.5pt{}\right#1\nulldelimiterspace=0pt \mathsurround=0pt$}}
\nc{\LEFT}[1]{\hbox{$\left#1\vbox to
  15.5pt{}\right.\nulldelimiterspace=0pt \mathsurround=0pt$}}
\nc{\kIGHT}[1]{\hbox{$\left.\vbox to
  15.5pt{}\right#1\nulldelimiterspace=0pt \mathsurround=0pt$}}

\nc{\bee}{{\bf E}} \nc{\bphi}{{\bf \Phi}}

\begin{document}

\title{Note on a conjecture of Braverman-Kazhdan}

\author{ G\'erard Laumon
\\ {\it Universit\'e Paris-Saclay et CNRS}
\\{\tt gerard.laumon@u-psud.fr } \and Emmanuel Letellier \\ {\it
  Universit\'e Paris Cit\'e, IMJ-PRG, CNRS} \\{\tt
 emmanuel.letellier@imj-prg.fr}
 }

\pagestyle{myheadings}

\maketitle

\begin{abstract}Given a connected reductive algebraic group $G$ defined over a finite field $\F_q$ together with a representation $\rho^\flat:G^\flat\rightarrow \GL_N$ of the dual group of $G$  (in the sense of Deligne-Lusztig), Braverman and Kazhdan \cite{BK} defined an exotic Fourier operator on the space of complex valued functions on $G(\F_q)$. Under some assumption on $\rho^\flat$, they gave a conjectural formula for the Fourier kernel which they prove when $G=\GL_n$ for some $n$. In these notes we give a simple proof of their conjecture for any $G$ without any assumption on $\rho^\flat$.

  \end{abstract}
\tableofcontents

\section{Introduction}\label{intro}

Let $G$ be a connected reductive algebraic groups over $\overline{\F}_q$ with geometric Frobenius $F:G\rightarrow G$. Let $G^\flat$ be a connected reductive algebraic group over $\overline{\F}_q$ with geometric Frobenius $F^\flat$ such that $(G,F)$ and $(G^\flat,F^\flat)$ are in duality (in the sense of Deligne-Lusztig \cite{DL}). Lusztig defined a partition of the set $\widehat{G^F}$ of all irreducible $\Q$-characters of $G^F$ (with $\ell$ a prime not dividing $q$) indexed by the set of $F^\flat$-stable semisimple conjugacy classes of $G^\flat$. We call the parts of this partition the \emph{Lusztig series} and we denote by ${\rm LS}(G)$ the set of Lusztig series of $(G,F)$. 

\bigskip

Let $(G',F)$ be a \emph{standard pair}, i.e. $G'$ is an $F$-stable Levi subgroup of some parabolic subgroup of $\GL_n$ where $F:\GL_n\rightarrow \GL_n$ is the standard Frobenius. We assume given  a morphism $\rho^\flat:G^\flat\rightarrow G'{^\flat}=G'$ which commutes with Frobenius. It induces a map between semisimple conjugacy classes fixed by Frobenius. Using Lusztig's parametrization we thus  get a map
$$
\t_\rho:{\rm LS}(G)\rightarrow {\rm LS}(G').
$$
Notice that in the case where $\rho^\flat$ is normal, i.e. its image is a normal subgroup of $G'$, there exists a morphism $\rho:G'\rightarrow G$ in duality with $\rho^\flat$ (see Proposition \ref{normal}) and in this case the map $\t_\rho$ is given by the pull-back functor  along the map $\rho^F:G'{^F}\rightarrow G^F$ (see Proposition \ref{func}). Although the morphism $\rho$ does not exists in general, we keep it in the notation $\t_\rho$.

\bigskip

{\bf Exotic Fourier operators}
\bigskip

\noindent Consider the space $\calC(G'{^F})$ of all functions $G'{^F}\rightarrow\Q$. Fix a non-trivial additive character $\psi:\F_q\rightarrow\Q^\times$ and consider the standard \emph{Fourier kernel}

$$
\phi^{G'}:=\psi\circ\Tr:G'{^F}\rightarrow\Q.
$$
The standard Fourier operator ${\bf F}^{G'}:\calC(G'{^F})\rightarrow\calC(G'{^F})$ is then defined by

$$
{\bf F}^{G'}(f)(x)=\sum_{y\in G'{^F}}\phi^{G'}(xy)f(y).
$$
We also consider the standard \emph{gamma function} $\gamma^{G'}_o:\widehat{G'{^F}}\rightarrow\Q$ defined by the formula

$$
\gamma^{G'}(\chi)=\sum_{g\in G'{^F}}\frac{\phi^{G'}(g)\chi(g)}{\chi(1)}.
$$
For any irreducible character $\chi$ of $G'{^F}$ we then have

$$
{\bf F}^{G'}(\chi)=\gamma^{G'}(\chi)\, \chi^\vee
$$
where $\chi^\vee$ denotes the dual character of $\chi$, and one can recover the Fourier kernel from the gamma function by the formula

\begin{equation}
\phi^{G'}=\sum_{\chi\in\widehat{G'{^F}}}\chi(1)\gamma^{G'}(\chi)\chi^\vee.
\label{Fourier-for}\end{equation}
The function $\gamma^{G'}$ is \emph{admissible}, i.e. it is constant on Lusztig series (see below Theorem \ref{commuteres}), and so induces a function on ${\rm LS}(G')$ which we still denote by $\gamma^{G'}$ as no confusion should arise.
Using the  morphism $\rho^\flat$, we can thus transfer $\gamma^{G'}$ into an admissible function on $\widehat{G^F}$ by putting

$$
\gamma^G_\rho:=c_{G,G'}\,\gamma^{G'}\circ\t_\rho,
$$
where $c_{G,G'}=\epsilon_G\epsilon_{G'}\, q^{v_G-v_{G'}}$ with $\epsilon_G=(-1)^{\F_q-\text{rank}(G)}$ and $v_G$ the dimension of the unipotent radical of a Borel subgroup of $G$.
\bigskip

Define then the corresponding Fourier kernel  $\phi^G_\rho:G^F\rightarrow\Q$ by the right hand side of the formula (\ref{Fourier-for}) with $\gamma^{G'}$ (resp. $G'$) replaced by $\gamma^G_\rho$ (resp. $G$) and denote by ${\bf F}^G_\rho:\calC(G^F)\rightarrow\calC(G^F)$ the corresponding \emph{exotic Fourier operator} with kernel $\phi^G_\rho$, i.e.

$$
{\bf F}^G_\rho(f)(x)=\sum_{y\in G^F}\phi^G_\rho(xy)f(y)
$$
for all $f\in\calC(G^F)$.
\bigskip

These Fourier transforms ${\bf F}^G_\rho$ were first studied by Braverman and Kazhdan in \cite{BK}. They conjectured an explicit formula for the kernel $\phi^G_\rho$.

 The aim of these notes is to provide a simple proof of their conjecture. Moreover we do not make any assumption on $\rho^\flat$ unlike in \cite{BK}.

\bigskip

{\bf Description of $\phi^G_\rho$ in terms of Deligne-Lusztig theory}
\bigskip

\noindent We will consider the quotient stack $[G/G]$ for the conjugation action and we will identify freely the space $\calC([G/G]^F)$ of functions on $[G/G]^F$ with the subspace of $\calC(G^F)$ of $G^F$-invariant functions (this is possible because $G$ is connected).
\bigskip

Fix a maximally split $F$-stable maximal torus $T$ of $G$ and denote by $N$ the normalizer  of $T$ in $G$. 

Recall that the $G^F$-conjugacy classes of $F$-stable maximal tori of $G$ are parametrized by the set $H^1(F,N)$ of $F$-conjugacy classes. For $\overline{w}\in H^1(F,N)$ and $w\in N$ a representative of $\overline{w}$, we denote by $T_w$ an $F$-stable maximal torus of $G$ in the $G^F$-class of $F$-stable maximal tori corresponding to $\overline{w}$. Then the Frobenius $F$ on $T_w$ corresponds to the Fobenius $F\circ w$ on $T$.

A way to put all together the $F$-stable maximal tori of $G$ is to consider the quotient stack $[T/N]$. Indeed

\begin{align*}
[T/N]^F&=\coprod_{\overline{w}\in H^1(F,N)}[T^{F\circ w}/N^{F\circ w}]\\
&\simeq\coprod_{\overline{w}\in H^1(F,N)}[T_w^F/N_w^F]
\end{align*}
where $N_w=N_G(T_w)$.

\bigskip

\noindent For any $F$-stable maximal torus $H$ of $G$, put 
$$
\gamma^H_\rho:=c_{H,G'}\,\gamma^{G'}\circ\t_{\rho_H}:\widehat{H^F}\rightarrow\Q
$$
where $\rho^\flat_H:H^\flat\rightarrow G'$ is the composition $H^\flat\hookrightarrow G^\flat\rightarrow G'$ and denote by $\phi^H_\rho:H^F\rightarrow\Q$ the associated kernel.

By Lemma \ref{remcru}, we have the following explicit description : let $T'$ be an $F$-stable maximal torus of $G'$ that contains the image of $H^\flat$ and denote by $\rho_H:T'\rightarrow H$ the morphism in duality with $H^\flat\rightarrow T'$ (see Remark \ref{rem1}). Then

$$
\phi^H_\rho=c_{H,T'}\, \rho_{H\, !}(\psi\circ\Tr|_{T'}).
$$

The above collection of kernels $\phi^H_\rho$ provides thus an explicit kernel

$$
\phi^{[T/N]}_\rho\in\calC([T/N]^F),
$$
and so a Fourier transform 

$$
{\bf F}^{[T/N]}_\rho:\calC([T/N]^F)\rightarrow\calC([T/N]^F).
$$
We denote by ${\bf F}^{[G/G]}_\rho:\calC([G/G]^F)\rightarrow\calC([G/G]^F)$ the restriction of  ${\bf F}^G_\rho$ to invariant functions.

In \S \ref{LusInd} we define an injective $\Q$-linear map $\I_{[T/N]}^G:\calC([T/N]^F)\hookrightarrow \calC([G/G]^F)$ in terms of Deligne-Lusztig induction. If $G$ is of type $A$ with connected center, then it is an isomorphism.

We have the following theorem (see Theorem \ref{propmain}).

\begin{theorem}The following diagram commutes

\begin{equation}
\xymatrix{\calC([T/N]^F)\ar[d]_{{\bf F}_\rho^{[T/N]}}\ar@{^{(}->}[rr]^{(\I_{[T/N]}^G\circ\epsilon)(v_G)}&&\calC([G/G]^F)\ar[d]^{{\bf F}_\rho^{[G/G]}}\\
\calC([T/N]^F)\ar@{^{(}->}[rr]^{\I_{[T/N]}^G}&&\calC([G/G]^F).}
\label{comdiag0}\end{equation}

In particular

\begin{equation}
\I_{[T/N]}^G(\phi_\rho^{[T/N]})=\phi_\rho^G.
\label{mainfor0I}\end{equation}
\end{theorem}

We prove the commutativity of the diagram (\ref{comdiag0}) using the spectral definition of ${\bf F}^{[T/N]}_\rho$ and ${\bf F}^{[G/G]}_\rho$ (i.e. the definition in terms of gamma functions). Since the function $\phi^{[T/N]}_\rho$ has an explicit construction, Formula (\ref{mainfor0I}) gives an explicit realization of $\phi^G_\rho$.

The fact that Formula (\ref{mainfor0I}) is a consequence of the commutativity of the diagram is explained in Remark \ref{remcom}(2).

\bigskip

{\bf Geometric realization of $\phi^G_\rho$}
\bigskip

\noindent For simplicity we assume that $G'=\GL_n$ and that $\rho^\flat$ restricts to a morphism $\rho^\flat:T^\flat\rightarrow T'$ where $T'$ is the maximal torus of $G'$ of diagonal matrices. Then let $\rho:T'\rightarrow T$ be a morphism in duality with $\rho^\flat$ (see above Proposition \ref{normal}). If we put $W:=N/T$, the morphism $\rho$ is naturally $W$-equivariant.

Consider the Artin-Schreier sheaf $\mathcal{L}_\psi$ on $\overline{\F}_q$ and put $\Phi^{T'}=\Tr^*(\mathcal{L}_\psi)$. We then consider the complex 

$$
\Phi^T_\rho:=\rho_ !\Phi^{T'}[{\rm dim}\, T']
$$
in the ``derived category'' $\dcb(T)$ of constructible $\ell$-adic sheaves on $T$. This complex is naturally $W$-equivariant, comes with a natural Weil structure $F^*\Phi^T_\rho\simeq\Phi^T_\rho$, and the two are compatible (see \S \ref{Prelim}). Moreover $\Phi^T_\rho$ is a perverse sheaf (not necessarily irreducible) and so we get a perverse sheaf $\Phi^{[T/N]}_\rho$ equipped with a Weil structure $\varphi^{[T/N]}_\rho:F^*\Phi^{[T/N]}_\rho\simeq\Phi^{[T/N]}_\rho$.

For a stack $\mathfrak{X}$ equipped with a geometric Frobenius, denotes by $\calM(\mathfrak{X};F)$ the category of perverse sheaves on $\mathfrak{X}$ equipped with a Weil structure.

Put

$$
(\Phi^G_\rho,\varphi^G_\rho):=\calI_{[T/N]}^G(\Phi^{[T/N]}_\rho,\varphi^{[T/N]}_\rho)
$$
where $\calI_{[T/N]}^G:\calM([T/N];F)\rightarrow\calM([G/G];F)$ is a geometric realization of $\I_{[T/N]}^G$ (see \S \ref{ind}).
\begin{theorem}We have

\begin{equation}
{\bf X}_{\Phi^G_\rho,\varphi^G_\rho}=\epsilon_G\, \phi^G_\rho
\label{BK}\end{equation}
where ${\bf X}_{\Phi^G_\rho,\varphi^G_\rho}$ is the characteristic function of $(\Phi^G_\rho,\varphi^G_\rho)$.
\end{theorem}
\bigskip

The complex $\calI_{[T/N]}^G(\Phi^{[T/N]})$ is in fact the $W$-invariant part of the induced complex $\Ind_T^G(\Phi^T_\rho)$, where $\Ind_T^G$ is the parabolic induction from $T$ to $G$, and so the formula (\ref{BK}) is the one conjectured by Braverman-Kazhdan \cite{BK} under some restriction on $\rho^\flat$ (see Remark \ref{remark} for more details). They proved their conjecture when $G$ is a general linear group \cite{BK}. In the same paper, under the same assumption on $\rho^\flat$, they also conjectured that the geometric Fourier transform
$$
\mathcal{F}^{[G/G]}_\rho:\dcb([G/G])\rightarrow\dcb([G/G])
$$
defined from the  kernel $\Phi^G_\rho$ commutes with $\Ind_T^G$, and observed that this geometric conjecture implies Formula (\ref{BK}). 
This geometric conjecture reduces to the following one.

\begin{conjecture}Let $B=TU$ be a Borel subgroup of $G$. The direct image with proper support of $\Phi^G_\rho$ along the map $
G\rightarrow G/U$ is supported on $T=B/U\hookrightarrow G/U$.
\label{GBK}\end{conjecture}
Conjecture \ref{GBK} was proved by Cheng and Ng\^o \cite{NC} when $G$ is a general linear group. Shortly after the first version of this paper was posted, Chen T.-H. \cite{C} proved it for an arbitrary connected reductive group  (under some restriction on the characteristic).
\bigskip

Conjecture \ref{GBK} interprets commutativity of Fourier with induction in terms of kernels while our 	approach to prove (\ref{BK}) is to interpret commutativity of Fourier with induction in terms of gamma functions which are more flexible (see equivalence of assertions (1) and (3) in Proposition \ref{equiv}).
\bigskip

Let us finally remark that we have a natural geometric version   $\calF^{[T/N]}_\rho:\calM([T/N])\rightarrow\calM([T/N])$ of ${\bf F}^{[T/N]}_\rho$ and that when $G$ is of type $A$ with connected center, then $\calI_{[T/N]}^G$ is an equivalence  $\calM([T/N])\simeq \calM([G/G])$ (see \cite{LL}). The desired geometric Fourier transform $\calF^{[G/G]}_\rho:\calM([G/G])\rightarrow\calM([G/G])$ can then be defined so that the geometric version of the diagram (\ref{comdiag0}) commutes.

\section{Reminders and notations}

Our base field is a finite field $\F_q$ with $q$ elements and we choose an algebraic closure $\overline{\F}_q$. We fix a prime $\ell$ which is different from the characteristic of $k$ and we fix an algebraic closure $\Q$ of the field of $\ell$-adic numbers. We choose an involution $\Q\rightarrow\Q$, $z\mapsto\overline{z}$ such that $\overline{z}=z^{-1}$ if $z$ is root of unity.

\subsection{Reductive groups}\label{reductive}

If $T$ is a torus defined over $\F_q$ with Frobenius $F$,  we have the multiplicative norm map $\Nr_{F^r/F}:T(\F_{q^r})\rightarrow T(\F_q)$, $x\mapsto xF(x)\cdots F^{r-1}(x)$.

Let $G$ be a connected reductive algebraic group over $\F_q$ with a geometric Frobenius $F:G\rightarrow G$ associated with some $\F_q$-structure on $G$ and denote by $\mathcal{L}=\mathcal{L}_G:G\rightarrow G$, $g\mapsto g^{-1}F(g)$ the Lang map. For a maximal torus $T$ of $G$ we denote by $X(T)$ the character group and by $Y(T)$ the co-character group. We denote by $Z_G$ the center of $G$ and  for any $x\in G$ we denote by $C_G(x)$ the centralizer of $x$ in $G$.

\bigskip

\textbf{Relative Weyl groups}

\bigskip

\noindent For a Levi subgroup $L$ of (some parabolic subgroup of) $G$ we put

$$
W_G(L):=N_G(L)/L,
$$
where $N_G(L)$ denotes the normalizer of $L$ in $G$. 


\begin{example} A Levi factor $L$ of a parabolic subgroup of $\GL_n$ is $\GL_n$-conjugate to $L_o=(\GL_{n_1})^{a_1}\times\cdots\times(\GL_{n_r})^{a_r}$ for some positive integers $a_1,\dots,a_r$ and $n_1>\cdots >n_r$ such that $\sum_{i=1}^ra_in_i=n$. 
Note that the symmetric group in $m$ letters $S_m$ acts on each $(\GL_s)^m$ as $w\cdot(g_1,\dots,g_m)=(g_{w^{-1}(1)},\dots,g_{w^{-1}(n)})$. Therefore we have an action of $S_{a_1}\times\cdots\times S_{a_r}$ on $L_o$ and 

$$N_{\GL_n}(L_o)\simeq L_o\rtimes (S_{a_1}\times\cdots\times S_{a_r}),$$
and so $W_{\GL_n}(L_o)\simeq S_{a_1}\times\cdots\times S_{a_r}$. 
\label{n}\end{example}

Recall that if we fix an $F$-stable Levi factor $L$ of some parabolic subgroup of $G$, then  the $G^F$-conjugacy classes of $F$-stable Levi factors (of some parabolic subgroup of $G$) that are conjugate under $G$ to $L$ are parametrized by the set $H^1(F,N_G(L))=H^1(F,W_G(L))$ of $F$-conjugacy classes of $W_G(L)$. For $w\in N_G(L)$ (or $w\in W_G(L)$), we then denote by $L_w$ a representative of the $G^F$-conjugacy class of $F$-stable Levi factors corresponding to the class of $w$ in $H^1(F,N)$. The pair $(L_w,F)$ is then isomorphic to the pair $(L,F\circ w)$ where $F\circ w$ is the Frobenius $t\mapsto F(w^{-1}tw)$ on $L$. More precisely, we have $L_w=g^{-1}Lg$ with $g\in G$ such that $gF(g^{-1})=w$.

\bigskip

We say that an $F$-stable maximal torus of $G$ is \emph{maximally split} if it is contained in some $F$-stable Borel subgroup of $G$.

\bigskip

{\bf Duality}
\bigskip

\noindent Let $T$ be a maximally split $F$-stable maximal torus of $G$ with corresponding Weyl group $W$ and let $B$ be an $F$-stable Borel subgroup of $G$ containing $T$. Consider another connected reductive group $G^\flat$ endowed with a Frobenius $F^\flat:G^\flat\rightarrow G^\flat$. Let $T^\flat$ be a maximally split $F^\flat$-stable maximal torus of $G^\flat$, $B^\flat$ an $F^\flat$-stable Borel subgroup containing $T^\flat$ and let $W^\flat$ the Weyl group of $G^\flat$ with respect to $T^\flat$. If there exists an isomorphism $\varphi:X(T)\rightarrow Y(T^\flat)$ which takes simple roots (with respect to $B$) to simple coroots (with respect to $B^\flat$) and which is compatible with the action of the Galois group ${\rm Gal}(\overline{\F}_q/\F_q)$, then we say that $(G^\flat,F^\flat)$ and $(G,F)$ are \emph{dual groups} \cite[Definition 5.21]{DL}. In particular, for two tori $S$ and $S^\flat$ with Frobenius $F$ and $F^\flat$, the pairs $(S,F)$ and $(S^\flat,F^\flat)$ are dual if there exists an isomorphism $X(S)\simeq Y(S^\flat)$ compatible with Galois group actions. 

The classification of connected reductive groups in terms of root data ensures the existence of a dual group $(G^\flat,F^\flat)$ for $(G,F)$ (unique up to isomorphism).  From now $(G^\flat,F^\flat)$ will denote a group in duality with $(G,F)$.  

\bigskip

As the functor $X$ is contravariant and the functor $Y$ covariant, we have a canonical anti-isomorphism $W\rightarrow W^\flat, w\mapsto w^\flat$ such that for all $\chi\in X(T)$ and $w\in W$, we have

 $$
 \varphi(w\cdot\chi)=w^\flat\cdot \varphi(\chi),
 $$
where $w\cdot\chi:=\chi\circ w$. 

It also satisfies
$$
F(w)^\flat=(F^\flat){^{-1}}(w^\flat),
$$
for all $w\in W$. 

The map $w\mapsto w^\flat$ defines a bijection $H^1(F,W)\rightarrow H^1(F^\flat,W^\flat)$ and so a bijection between the set of $G^F$-conjugacy classes of $F$-stable maximal tori of $G$ and the set of $F^\flat$-stable maximal tori of $G^\flat$.
\bigskip

Let $L$ be an $F$-stable Levi factor of some parabolic subgroup of $G$ and $S$ be an $F$-stable maximal torus of $L$ maximally split. There exists an $F^\flat$-stable Levi factor $L^\flat$ of some parabolic subgroup of $G^\flat$ together with an $F^\flat$-stable  maximal torus $S^\flat$ of $L^\flat$ maximally split such that $(L,S,F)$ is in duality with $(L^\flat,S^\flat,F^\flat)$. More precisely, if $S$ is of the form $T_w$ for some $w\in W$ then $S^\flat$ is of the form $T^\flat_{w^\flat}$. 
\bigskip

\subsection{Lusztig induction} \label{LusInd}

\bigskip

Let $L$ be an $F$-stable Levi subgroup of some parabolic subgroup $P$ of $G$ (which may not be $F$-stable) and denote by $U_P$ the unipotent radical of $P$. The variety $\mathcal{L}^{-1}(U_P)$ is equipped with an action of $G^F$ by left mutliplication and with an action of $L^F$ by right multiplication. These actions induce actions of $G^F$ and $L^F$ on the compactly supported $\ell$-adic cohomology groups $H_c^i(\mathcal{L}^{-1}(U_P),\Q)$.
  
 For any character $\pi$ of $L^F$, the $\Q$-vector space

 $$
 M_{L\subset P}^i(\pi):=H^i_c(\mathcal{L}^{-1}(U_P),\Q)\otimes_{\Q[L^F]}V_\pi
 $$
 is thus a $G^F$-module and we denote by $R_{L\subset P}^G(\pi)$ the virtual character of $G^F$ defined by
 
 $$
 R_{L\subset P}^G(\pi)(g):=\sum_i(-1)^i\Tr\left(g, M_{L\subset P}^i(\pi)\right),
 $$
 for all $g\in G$.
 
Explicitly \cite[\S 10.1]{DM} we have 
 
 \begin{equation}
 R_{L\subset P}^G(\pi)(g)=\frac{1}{|L^F|}\sum_{l\in L^F}\Tr\,\left((g,l)\,,\, H_c^*(\mathcal{L}^{-1}(U_P),\Q)\right)\, \pi(l^{-1}),
 \label{DM11.2}\end{equation}
for all $g\in G^F$, where $H_c^*(\mathcal{L}^{-1}(U_P),\Q):=\sum_i(-1)^i H_c^i(\mathcal{L}^{-1}(U_P),\Q)$.

 Recall that the operator $R_{L\subset P}^G$ does not depend on the choice of the parabolic subgroup having $L$ as a Levi factor expect in some cases when $q=2$  where it is still a conjecture. The independence is a consequence of Mackey's formula see \cite[\S 9.2]{DM}.  From now we will denote simply by $R_L^G$ this operator.
 
 We have the following basic properties \cite[Propositions 9.1.7, 9.1.8]{DM}.
 
 \begin{proposition} (i) If $L\subset M$ is an inclusion of Levi subgroups, then $R_L^G=R_M^G\circ R_L^M$.

 \noindent (ii) If $\pi$ is a character of $L^F$, then $R_L^G(\pi^\vee)=R_L^G(\pi)^\vee$.
 \end{proposition}
 
We have the following result \cite[Corollary 7.7]{DL}.

\begin{theorem}[Deligne-Lusztig]Any irreducible character  of $G^F$ appears in some virtual character $R_T^G(\theta)$ for some $F$-stable maximal torus $T$ of $G$ and some linear character $\theta:T^F\rightarrow\Q^\times$.
\label{DL}\end{theorem}

\subsection{Lusztig series}

We have the following proposition \cite[(5.21.5)]{DL}.

\begin{proposition} (i) There exists a bijective correspondence between the set of $G^F$-conjugacy classes of pairs $(T,\theta)$, where $T$ is an $F$-stable maximal torus of $G$ and $\theta\in\widehat{T^F}$, and the set of $G^\flat{^{F^\flat}}$-conjugacy classes of pairs $(T^\flat,s)$, where $T^\flat$ is an $F^\flat$-stable maximal torus  of $G^\flat$ and $s\in T^\flat{^{F^\flat}}$. 

\noindent (ii) If $(T,F)$ and $(T^\flat,F^\flat)$ are in duality, then $\widehat{T^F}\simeq T^\flat{^{F^\flat}}$.

\label{DL1}\end{proposition}

The correspondence (i) and the isomorphism (ii) of the proposition depends on the choice of the  isomorphism $\overline{\F}_q^\times\simeq (\mathbb{Q}/\Z)_{p'}$ and the embedding $\overline{\F}_q^\times\hookrightarrow\Q^\times$. Indeed, to construct the correspondence between characters and $F$-stable points of tori from an isomorphism $X(S)\simeq Y(S^\flat)$, we relate characters of $S^F$ with $X(S)$ and $F^\flat$-stable points of $S^\flat$ with $Y(S^\flat)$ as follows. The choice of an isomorphism $\overline{\F}_q^{\times}\simeq (\mathbb{Q}/\Z)_{p'}$ defines a surjective group homomorphism $Y(S)\rightarrow S^F$, $y\mapsto \Nr_{F^n/F}(y(\zeta))$ where $n$ is such that $S$ is split over $\F_{q^n}$ and $\zeta$ is the $(q^n-1)$-th root of unity corresponding to $1/(q^n-1)\in(\mathbb{Q}/\Z)_{p'}$.  The choice of the embedding $\overline{\F}_q^\times\hookrightarrow\Q^\times$ defines a surjective group homomorphism $X(S)\rightarrow \widehat{S^F}$ by restricting a character $S\rightarrow \overline{\F}_q^\times$ to $S^F$.
\bigskip

\noindent If $H$ is an $F$-stable maximal torus of $G$ and $\eta\in\widehat{H^F}$, we call $(H,\eta)$ a \emph{Deligne-Lusztig pair} of $(G,F)$ (DL pair of $G$ for short). We say that two DL pairs $(T,\theta)$ and $(T',\theta')$ of $G$ are \emph{geometrically conjugate} if there exists some positive integer $n$, some $g\in G^{F^n}$ such that 

$$
gTg^{-1}=T'\,\,\text{ and }\hspace{.5cm}\theta\circ \Nr_{F^n/F}(t)=\theta'\circ \Nr_{F^n/F}(gtg^{-1}),
$$
for all $t\in T^{F^n}$. 

We have the following proposition \cite[Proposition 5.22]{DL}.

\begin{proposition}Geometric conjugacy classes of DL pairs $(T,\theta)$ are in one-to-one correspondence with $F^\flat$-stable conjugacy classes of semi-simple elements of $G^\flat$.
\label{DL2}\end{proposition}

When the $G^F$-conjugacy class of the  DL pair $(T,\theta)$ corresponds to the $G^\flat{^{F^\flat}}$-conjugacy class of a pair $(T^\flat,s)$ (see Proposition \ref{DL1}(i))  we will write  sometimes $R_{T^\flat,s}^G$ instead of $R_T^G(\theta)$.

\begin{theorem}[Deligne-Lusztig] $R_{T^\flat,s}^G$ and $R_{T'{^\flat},s'}^G$ have no commun irreducible constituent unless $s$ and $s'$ are $G^\flat$-conjugate.
\label{DL3}\end{theorem}

\begin{remark}When the center of $G$ is connected, the centralizers of the semisimple elements of $G^\flat$ are all connected \cite[Remark 11.2.2(ii)]{DM} and so $R_{T^\flat,s}^G$ and $R_{T'{^\flat},s'}^G$ have no commun irreducible constituent unless $s$ and $s'$ are $G^\flat{^{F^\flat}}$-conjugate.

\label{connectedcenter}\end{remark}

A \emph{geometric Lusztig series} (or \emph{Lusztig series} for short) of $(G,F)$  associated to the geometric conjugacy class of some DL pair $(T,\theta)$ of $G$ is the set of all irreducible characters of $\widehat{G^F}$ which appear non-trivially in some $R_{T'}^G(\theta')$ where $(T',\theta')$ is geometrically conjugate to $(T,\theta)$. Thanks to Theorem \ref{DL}, any irreducible character  of $G^F$ belongs to a Lusztig series and thanks to Theorem \ref{DL3}, the Lusztig series are disjoint and so form a partition of $\widehat{G^F}$ which is parametrized by the $F^\flat$-stable semisimple conjugacy classes of $G^\flat$. 

Let $(T,\theta)$ be a DL pair of $G$ and $s\in G^\flat{^{F^\flat}}$ be a corresponding semisimple element. We denote either by $\mathcal{E}_G(T,\theta)$ or $\mathcal{E}_G(s)$  the Lusztig series associated with the geometric conjugacy class of $(T,\theta)$. For $\pi\in\widehat{G^F}$ we will also denote by $\mathcal{E}_G(\pi)$ the Lusztig series which contains $\pi$.

We denote by ${\rm LS}(G)$ the set of Lusztig series of $(G,F)$.

\begin{remark}Given two irreducible characters $\pi$ and $\pi'$  of $G^F$, say that
$$
\pi\sim\pi'
$$
if there exists a DL pair $(T,\theta)$ such that $\pi$ and $\pi'$ both appear in $R_T^G(\theta)$. When the center of $G$ is connected, we see that the equivalence classes for the transitive closure of $\sim$ coincide with Lusztig series. Indeed, the unipotent characters form the Lusztig series which corresponds to $s=1$. On the other hand, as the trivial character of $G^F$ appears in all $R_T^G(1)$, where $T$ runs over the $F$-stable maximal tori, any two unipotent characters of $G^F$ are equivalent.  We deduce the case of arbitrary characters from the unipotent case using Lusztig's Jordan decomposition of irreducible characters \cite[Theorem 11.5.1]{DM} together with the fact that the centralizers of the semisimple elements of $G^\flat$ are all connected (since the center of $G$ is connected).

\label{ratLS}\end{remark}

\section{Preliminaries}

\subsection{Gamma functions on finite groups}\label{defgamma}

Fix a finite group $G$ and denote by $\widehat{G}$ the set of irreducible $\Q$-characters of $G$ and by $\calC(G)$ the $\Q$-vector space of all functions $G\rightarrow\Q$.

The action of $G$ on itself by right and left multiplication makes $\calC(G)$ into a $G\times G$-module by

$$
\left((g,h)\cdot f\right) (x)=f(g^{-1}xh)
$$
for all $f\in\calC(G)$, $g,h,x\in G$. It decomposes into irreducible as

$$
\mathcal{C}(G)=\bigoplus_{\pi\in\widehat{G}} \calC_\pi(G),
$$
where $\mathcal{C}_\pi(G)$ is the subspace generated by $\{(g,1)\cdot\pi\,|\, g\in G\}$. In fact $\calC_\pi(G)\simeq V_\pi\boxtimes V_\pi^\vee$ where  $V_\pi$ is an irreducible representation with character $\pi$ and $V_\pi^\vee$ is the dual representation.
\bigskip

Denote by $\calC(G)^\iota$ the $G\times G$-module with underlying space $\calC(G)$ and with $G\times G$-action given by

$$
(g,h)*f=(h,g)\cdot f.
$$

If  ${\bf F}^G:\mathcal{C}(G)\rightarrow\mathcal{C}(G)^\iota$ is a morphism  of $G\times G$-modules, we have a function $\gamma^G:\widehat{G}\rightarrow\Q$ (which we call a \emph{gamma function}) such that for every $\pi\in\widehat{G}$ we have

\begin{equation}
{\bf F}^G(\pi)=\gamma^G(\pi)\, \pi^\vee.
\label{spec}\end{equation}

The function $\gamma^G$ determines completely ${\bf F}^G$.

\begin{proposition}Let  ${\bf F}^G:\mathcal{C}(G)\rightarrow \mathcal{C}(G)$ be a $\Q$-linear map. The following assertions are equivalent.

\noindent (1) ${\bf F}^G:\mathcal{C}(G)\rightarrow\mathcal{C}(G)^\iota$ is a morphism of $G\times G$-modules.

\noindent (2) The function ${\bf F}^G$ is given by a kernel, i.e. there exists a central function $\phi^G\in\mathcal{C}(G)$ such that

\begin{equation}
{\bf F}^G(f)(g)=\sum_{h\in G}\phi^G(gh)f(h),
\label{ker}\end{equation}
for all $f\in\mathcal{C}(G^F)$ and $g\in G^F$.

\label{kernel}\end{proposition}

\begin{proof} If (1) holds, the function $\phi^G$ in (2) is given by the formula

\begin{equation}
\phi^G(g)=\sum_{\pi\in\widehat{G}}\gamma^G(\pi)\pi(1)\overline{\pi(g)},
\label{inversegam}
\end{equation}
for all $g\in G$.

\end{proof}

\begin{remark}From (\ref{spec}) and (\ref{ker}) we see that

\begin{equation}
\gamma^G(\pi)=\pi(1)^{-1}\sum_{g\in G}\phi^G(g)\pi(g^{-1}),
\label{inversephi}\end{equation}
for all $\pi\in\widehat{G}$.
\end{remark}

\noindent From the above discussion we see that it is equivalent to give oneself :
\bigskip

$\bullet$ a morphism ${\bf F}^G:\calC(G)\rightarrow\calC(G)^\iota$  of $G\times G$-modules,

$\bullet$ a function $\gamma^G:\widehat{G}\rightarrow\Q$ (\text{gamma function}),

$\bullet$ a central function $\phi^G:G\rightarrow\Q$ (\text{kernel}).

\subsection{Quotient stacks over finite fields}\label{QSFF}

Assume that $Z$ is a finite set on which a finite group $H$ acts on the right. By notation abuse and when the context is clear, we denote the map $Z\rightarrow Z$, $z\mapsto z\cdot h$ by $h$. Denote by $[Z/H]$ the category of $H$-equivariant maps $H\rightarrow Z$ where $H$ acts on itself by right translation and denote by $\overline{[Z/H]}$ the set of isomorphism classes of $[Z/H]$. 
A function $[Z/H]\rightarrow\Q$ is a function on objects which is constant on isomorphism classes, i.e. it is a $\Q$-valued function on $\overline{[Z/H]}$. We will often identify the $\Q$-space $\calC([Z/H])$ of  functions $[Z/H]\rightarrow\Q$  with the space $\calC(Z)^H$ of $H$-invariant functions on $Z$. We  consider on $\calC([Z/H])$ the inner form

\begin{align*}
(f,h)_{[Z/H]}:&=\sum_{[x]\in \overline{[Z/H]}}\frac{1}{|{\rm Aut}([x])|}f([x])\overline{h([x])}\\
&=\frac{1}{|H|}\sum_{x\in Z}f(x)\overline{h(x)}
\end{align*}
for all $f,h\in\calC([Z/H])$.

If $X$ is a $\overline{\F}_q$-scheme on which a $\overline{\F}_q$-algebraic group $G$ acts on the right, we denote by $[X/G]$ the quotient stack. We assume that $X$, $G$ and the action of $G$ on $X$ are defined over $\F_q$ and we denote by  $F$ the associated geometric Frobenius on $X, G$ and $[X/G]$. Recall that $[X/G]$ is defined as the functor

$$
{\rm Sch}/_{\overline{\F}_q}\longrightarrow {\rm Groupoids}
$$
sending a $\overline{\F}_q$-scheme $S$ to the groupoid of diagrams
$$
\xymatrix{E\ar[r]\ar[d]&X\\
S&}
$$
where the vertical arrow is a $G$-torsor over $S$ and the horizontal arrow is a $G$-equivariant morphism.
\bigskip

Note that the $G$-torsors over $\F_q$ are parametrized by the set $H^1(F,G)=H^1(F,G/G^o)$ of $F$-conjugacy classes on $G$, and so the $\F_q$-points of $[X/G]$ decomposes as the disjoint union

\begin{equation}
[X/G](\F_q)=[X/G]^F\simeq \coprod_{\overline{h}\in H^1(F,G)}[X^{F\circ h}/G^{F\circ h}]
\label{decomp11}\end{equation}
where $h\in G$ denotes a representative of $\overline{h}$ and $G^{F\circ h}=\{g\in G\,|\, F(h^{-1}gh)=g\}$.

Notice that if $G$ is connected, then $H^1(F,G)$ is trivial and so $[X/G]^F=[X^F/G^F]$.

\bigskip

\begin{example}Let $T$ be the torus of diagonal matrices in $\SL_2$ and let $W$ be the Weyl group of $\SL_2$ with respect to $T$. We have $[T/W]=[\mathbb{G}_m/\mu_2]$ where $\mu_2$ acts on $\mathbb{G}_m$ as $x\mapsto x^{-1}$. There are two torsors on ${\Spec}(\F_q)$ which are ${\rm Spec}(\F_{q^2})\rightarrow{\rm Spec}(\F_q)$ (with action of $\mu_2$ on ${\rm Spec}(\F_{q^2})$ given by the Frobenius) and the trivial one ${\rm Spec}(\F_q\times\F_q)\rightarrow{\rm Spec}(\F_q)$ (where $\mu_2$ acts on ${\rm Spec}(\F_q\times\F_q)$ by exchanging the coordinates). Then

$$
[\mathbb{G}_m/\mu_2](\F_q)=[\mathbb{F}_q^\times/\mu_2]\sqcup [\{x\in\mathbb{F}_{q^2}^\times\,|\, x^q=x^{-1}\}/\mu_2].
$$
\end{example}
\bigskip

From the decomposition (\ref{decomp11}) we have

$$
\calC([X/G]^F)=\bigoplus_{\overline{h}\in H^1(F,G)}\calC([X^{F\circ h}/G^{F\circ h}]).
$$
For $f\in\calC([X/G]^F)$ and $\overline{h}\in H^1(F,G)$ denote by $f_{\overline{h}}$ the coordinate of $f$  in $\calC([X^{F\circ h}/G^{F\circ h}])$. We consider on $\calC([X/G]^F)$ the inner product

\begin{align*}
\left(a,b\right)_{[X/G]^F}:&=\sum_{[x]\in \overline{[X/G]^F}}\frac{1}{|{\rm Aut}([x])|}\,a(x)\overline{b(x)}\\
&=\sum_{\overline{h}\in H^1(F,G)}\left(a_z,b_z\right)_{[X^{F\circ h}/G^{F\circ h}]}.
\end{align*}

For a function $f\in\calC([X/G]^F)$ and $n\in\mathbb{Z}$, we let $f(n)$ be the function which takes the values $q^{-n}f(x)$ at any $x\in[X/G]^F$.

\subsection{Lusztig induction on $[T/N]$}

Let $T$ be an $F$-stable maximal torus of $G$ with normalizer $N=N_G(T)$ and Weyl group $W=N/T$. 
\bigskip

\textbf{The quotient stack $[T/N]$}
\bigskip

\noindent By  \S \ref{QSFF}  the $\F_q$-points of $[T/N]$ decomposes as the disjoint union

\begin{equation}
[T/N](\F_q)=[T/N]^F=\coprod_{\overline{w}\in H^1(F,N)}[T^{F\circ w}/N^{F\circ w}].
\label{decomp}\end{equation}
We thus have

$$
\calC([T/N]^F)=\bigoplus_{\overline{w}\in H^1(F,N)}\calC([T^{F\circ w}/N^{F\circ w}]).
$$
Since the action of $F\circ w$ on $T$ corresponds to the action of $F$ on $T_w$, we have

\begin{equation}
\calC([T/N]^F)=\bigoplus_{\overline{w}\in H^1(F,N)}\calC([T_w^F/N_w^F])
\label{T/N}\end{equation}
where $N_w=N_G(T_w)$.  The inner product on $\calC([T/N]^F)$ is then given by

\begin{align*}
\left(a,b\right)_{[T/N]^F}&=\sum_{[x]\in \overline{[T/N]^F}}\frac{1}{|{\rm Aut}([x])|}\,a(x)\overline{b(x)}\\
&=\sum_{\overline{w}\in H^1(F,N)}\frac{1}{|W_w^F|}\,\left(a_w,b_w\right)_{[T_w/T_w]^F}.
\end{align*}
for any two functions $a,b\in \calC([T/N]^F)$, where $W_w=N_w/T_w$.

\begin{remark}Notice that

$$
\calC([T/W]^F)=\calC([T/N]^F)
$$
but it is more natural to work with $N$ instead of $W$ (see \cite{LL}). Working with $W$ instead of $N$ modifies the above inner form by a factor $|T_w^F|$.
\end{remark}
\bigskip

{\bf Lusztig induction and restriction} 

\bigskip

\noindent Let $L$ be an $F$-stable Levi subgroup of some parabolic subgroup $P$ of $G$ (which may not be $F$-stable). In the following we identify $\calC([G/G]^F)$ with $\calC(G^F)^{G^F}$.
\bigskip

 The Lusztig induction $\pi\mapsto R_{L\subset P}^G(\pi)$ defined on characters extends linearly to a $\Q$-linear map $R_L^G:\mathcal{C}([L/L]^F)\rightarrow\mathcal{C}([G/G]^F)$. Explicitly
 
 \begin{equation}
 R_{L\subset P}^G(f)(g)=\frac{1}{|L^F|}\sum_{l\in L^F}\Tr\,\left((g,l)\,,\, H_c^*(\mathcal{L}^{-1}(U_P),\Q)\right)\, f(l^{-1}),
 \label{DM11.2}\end{equation}
for all $g\in G^F$ and $f\in\calC([L/L]^F)$.
\bigskip

We define the Lusztig restriction ${^*}R_L^G:\mathcal{C}([G/G]^F)\rightarrow\mathcal{C}([L/L]^F)$ by the formula \cite[\S 10.1]{DM}

$$
{^*}R_L^G(f)(l)=\frac{1}{|G^F|}\sum_{g\in G^F}\Tr\,\left((g,l)\,,\, H_c^*(\mathcal{L}^{-1}(U_P),\Q)\right)\, f(g^{-1}),
$$
for all $l\in L^F$ and $f\in\calC([G/G]^F)$.

\begin{proposition}\cite[Definition 9.1.3, Proposition 9.1.6]{DM}
The two operators ${^*}R_L^G$ and $R_L^G$ are adjoint to each other with respect to $(\,,\,)_{[G/G]^F}$ and $(\,,\,)_{[L/L]^F}$.
\end{proposition}
\bigskip

\begin{remark}If $f\in\calC([G/G]^F)$, then ${^*}R^G_L(f)$ is $N_G(L)^F$-invariant.
\label{rem-equi}\end{remark}

Define the induction

$$
\I_{[T/N]}^G:\calC([T/N]^F)\rightarrow\calC([G/G]^F),\hspace{1cm}f\mapsto\sum_{\overline{w}\in H^1(F,N)}\frac{1}{|W_w^F|}R_{T_w}^G(f_w).
$$
Notice that

$$
\I_{[T/N]}^G(f)=\frac{1}{|W|}\sum_{w\in W}R_{T_w}^G(f_w).
$$
The restriction is defined by

$$
{^*}\I^G_{[T/N]}:\calC([G/G]^F)\rightarrow\calC([T/N]^F),\hspace{1cm}h\mapsto \sum_{\overline{w}\in H^1(F,N)}{^*}R^G_{T_w}(h).
$$
From the above proposition we have the following result.

\begin{lemma}For any $f\in\calC([T/N]^F)$ and $h\in\calC([G/G]^F)$,we have

$$
\left({^*}\I_{[T/N]}^G(h),f\right)_{[T/N]^F}=\left(h,\I_{[T/N]}^G(f)\right)_{[G/G]^F}.
$$
\end{lemma}

Say that a function in $\calC([G/G]^F)$ is \emph{uniform} if it is a linear combination of Deligne-Lusztig characters $R_{T_w}^G(\theta)$ for various $w\in W$ and linear characters $\theta$ of $T_w^F$. We denote by $\calC([G/G]^F)_{\rm unif}$ the $\Q$-subspace of $\calC([G/G]^F)$ of uniform functions.

\begin{theorem}The map $\I_{[T/N]}^G$ induces an isomorphism

$$
\calC([T/N]^F)\simeq \calC([G/G]^F)_{\rm unif}
$$
with inverse given by the restriction of ${^*}\I_{[T/N]}^G$ to uniform functions.

\end{theorem}

\begin{proof}This is a consequence of Mackey's formula for Deligne-Lusztig induction.

\end{proof}

\begin{remark}If $G$ is of type $A$ with connected center, then $\calC([G/G]^F)_{\rm unif}=\calC([G/G]^F)$ and so in this case we have

$$
\calC([T/N]^F)\simeq \calC([G/G]^F)
$$
\end{remark}

\subsection{Lusztig induction and Lusztig series}

In this section we prove that Lusztig induction from a Levi induces a map between Lusztig series. More precisely we have the following result.

\begin{proposition}Let $L$ be an $F$-stable Levi factor of some parabolic subgroup of $G$ and let $\pi\in\mathcal{E}_L(T,\theta)=\mathcal{E}_L(s)$. Then any irreducible constituent of $R_L^G(\pi)$ belongs to $\mathcal{E}_G(T,\theta)=\mathcal{E}_G(s)$. Therefore the functor $R_L^G$ induces a map $\t_L^G:{\rm LS}(L)\rightarrow{\rm LS}(G)$.
\label{remLus}\end{proposition}

\begin{proof}Let $P=LU_P$ and $B=TU$ be a parabolic subgroup and Borel subgroup of $G$ such that $B\subset P$ and consider the Borel subgroup $B_L=B\cap L=TV$ of $L$. Assume that $\pi$ is an irreducible constituent of $R_T^L(\theta)$. Then $\pi$ is an irreducible constituent of some $H_c^i(\mathcal{L}^{-1}(V),\Q)\otimes_{\Q[T^F]}\theta$. On the other hand, for all non-negative integer $k$,  we have (transitivity of Lusztig  induction)

$$
H_c^k(\mathcal{L}^{-1}(U),\Q)\otimes_{\Q[T^F]}\theta\simeq \bigoplus_{i+j=k}\left(H_c^i(\mathcal{L}^{-1}(U_P),\Q)\otimes_{\Q[L^F]}H_c ^j(\mathcal{L}^{-1}(V),\Q)\right)\otimes_{\Q[T^F]}\theta.
$$
Therefore any irreducible consituent of $R_L^G(\pi)$ appears in some $H_c^k(\mathcal{L}^{-1}(U),\Q)\otimes_{\Q[T^F]}\theta$ for some $k$. Note that a priori $\alpha$ could also appear in other cohomology groups and we could have cancellation in $R_T^G(\theta)$. The proposition 11.1.3 of \cite{DM} says that $\alpha$ appears at least in some $R_{T'}^G(\theta')$ with $(T',\theta')$ in the geometric conjugacy class of $(T,\theta)$.
\end{proof}

\subsection{Morphisms in duality and Lusztig series}

A morphism $f:T'\rightarrow T$ of tori  induces a morphism $Y(T')\rightarrow Y(T)$ between the co-character groups and so a map $X(T'{^\flat})\rightarrow X(T^\flat)$ between the character groups of the dual tori. Since the contravariant functor $X$ is fully faithful, we get a morphism $f^\flat:T^\flat\rightarrow T'{^\flat}$. If $f$ commutes with Frobenius then so does $f^\flat$.

\begin{remark}In the case where $T'$ is the maximal torus $\T_N$ of diagonal matrices of $\GL_N$, the morphism $f^\flat$ is constructed as follows. The morphism $f:T_N\rightarrow \T$ is of the form $(t_1,\dots, t_N)\mapsto \alpha_1(t_1)\cdots\alpha_N(t_N)$ for some cocharacters $\alpha_1,\dots,\alpha_N$ of $T$. Regarding now the $\alpha_i$'s as characters of $T^\flat$ via the isomorphism $X(T^\flat)\simeq Y(T)$, we obtain  $f^\flat:T^\flat\rightarrow\T_N^\flat=\T_N$, $t\mapsto (\alpha_1(t),\dots,\alpha_N(t))$.
\label{rem1}\end{remark}

More generally, consider a morphism $f:H'\rightarrow H$ of connected reductive algebraic groups defined over $\F_q$ which is \emph{normal} (i.e. the image of $H'$ is a normal subgroup of $H$).

\begin{proposition} There exists a normal  morphism $f^\flat:H^\flat\rightarrow H'{^\flat}$ defined over $\F_q$ which extends any morphism $T^\flat\rightarrow T'{^\flat}$ obtained by duality from the restriction $T'\rightarrow T$ of $f$ to maximal tori. 
\label{normal}\end{proposition}

\begin{proof}To see this, we are reduced to the case where $f$ is a surjective morphism or the inclusion of a closed connected normal subgroup. First of all recall that any connected reductive group $G$ is the almost-direct product  of the connected component of its center and a finite number of quasi-simple groups $G_1,\dots,G_r$, i.e. the product map $Z_G^o\times G_1\times\cdots\times G_r\rightarrow G$ is an isogeny (that is a  surjective homomorphism with finite kernel). Therefore if $N$ is a closed connected normal subgroup of $H$, then there exists a closed connected normal subgroup $S$ of $H$ such that the $H$ is the almost-direct product of $S$ and $N$. The isogeny $S\times N\rightarrow H$ induces an isogeny between the root data and so an isogeny between the dual root data. By the isogeny theorem (see for instance \cite[Theorem 23.9]{Milne}) we thus get an isogeny $H^\flat\rightarrow (S\times N)^\flat\simeq S^\flat\times N^\flat$. Composing this isogeny with the projection $S^\flat\times N^\flat\rightarrow N^\flat$ we get the required morphism $H^\flat\rightarrow N^\flat$. 

We now assume that $f$ is surjective and denote by $S$ the kernel of $f$. As $f$ factorizes through the isogeny $H'/S^o\rightarrow H'/S$, we may assume that $S$ is connected. By the above discussion, the inclusion $i: S\hookrightarrow H'$  induces a surjective morphism $i^\flat:H'{^\flat}\rightarrow S^\flat$. Also if $T'$ denotes a maximal torus of $H'$,  $T_S:=T'\cap S$ and $T:=f(T')$, then we have an exact sequence of tori $1\rightarrow T_S\rightarrow T'\rightarrow T\rightarrow 1$ and so an exact sequence
$1\rightarrow T^\flat\rightarrow T'{^\flat}\rightarrow (T_S)^\flat\rightarrow 1$. Therefore ${\rm Ker}(i^\flat)\cap T'{^\flat}\simeq T^\flat$ is connected from which we deduce that $K^\flat:={\rm Ker}(i^\flat)$ is also connected. The map $T^\flat\rightarrow T'{^\flat}$ induces an isomorphism between the root data of $H^\flat$ and $K^\flat$ and so extends to an isomorphism $H^\flat\simeq K^\flat$ by the isogeny theorem.
\end{proof}
\bigskip

Let $f^\flat:H^\flat\rightarrow H'{^\flat}$ be a morphism in duality with $f$. We denote by $f^F:H'{^F}\rightarrow H^F$ the induced group homomorphism on rational points. 

\begin{proposition}The pull back functor $(f^F)^*:{\rm Rep}(H^F)\rightarrow{\rm Rep}(H'{^F})$, $\alpha\mapsto\alpha\circ f^F$ between categories of finite dimensional representations (over $\Q$) induces a map between the sets of Lusztig series. More precisely, if $\alpha\in\mathcal{E}_H(s)$, with $s\in H^\flat{^{F^\flat}}$, then any irreducible constituent of $\alpha\circ f^F$ belongs to   $\mathcal{E}_{H'}(f^\flat(s))$.
\label{func}\end{proposition}

\begin{proof}The statement is clear if both $H'$ and $H$ are direct products of a torus by a quasi-simple group. Since both $H'$ and $H$ are such direct products up to central isogeny we are reduced to prove the proposition for  $f:H'\rightarrow H$ a central isogeny. Let $T'$ be an $F$-stable maximal torus of a Borel subgroup $B'=T'U'$ of $H'$ and let $B=TU$ be the image of $B'$ by $f$ with $f(T)=T'$. By base change, the map $f$ induces a finite surjective map

$$
\tilde{f}:\coprod_{z\in {\rm Ker}(f)}\mathcal{L}_{H'}^{-1}(zU')\rightarrow\mathcal{L}_H^{-1}(U).
$$
Via $f$, the groups $H'{^F}$ and $T'{^F}$ act on $\mathcal{L}_H^{-1}(U)$ and  $\tilde{f}$ is invariant under these actions. We thus get for all $i$ an inclusion $H_c^i(\mathcal{L}_H^{-1}(U),\Q)\hookrightarrow \bigoplus_{z\in{\rm Ker}(f)}H_c^i(\mathcal{L}_{H'}^{-1}(zU'),\Q)$ of $H'{^F}\times T'{^F}$-modules. Expressing any $z\in{\rm Ker}(f)$ in the form $t'F(t'{^{-1}})$ for some $t'\in T'$ yields an $G'{^F}\times T'{^F}$-equivariant  isomorphism $\mathcal{L}_{H'}^{-1}(zU')\rightarrow\mathcal{L}_{H'}^{-1}(U')$. Therefore if $\theta\in \widehat{T^F}$ and $\rho$ is a representation of $H^F$ appearing in $H_c^i(\mathcal{L}_H^{-1}(U),\Q)_\theta$, the $H^F$-submodule of $H_c^i(\mathcal{L}_H^{-1}(U),\Q)$ on which $T^F$ acts by $\theta$, then $(f^F)^*(\rho)$ appears in $H_c^i(\mathcal{L}_{H'}^{-1}(U'),\Q)_{(f^F)^*(\theta)}$.
 \end{proof}

\subsection{Functoriality}\label{Functoriality}
\bigskip

\noindent Let $G$ and $G'$ be two connected reductive algebraic groups defined over $\F_q$ and, by notation abuse, use the letter $F$ for the corresponding geometric Frobenius on  $G$ and $G'$. Let $\rho^\flat:G^\flat\rightarrow G'{^\flat}$ be an algebraic morphism which commutes with Frobenius $F^\flat$. The \emph{functoriality principle} predicts a map $\t_\rho$ from certain ``packets" of irreducible representations of $G^F$ to ``packets" of irreducible representations of $G'{^F}$. 
\bigskip

\noindent The packets we consider are the Lusztig series and

$$
\t_\rho:{\rm LS}(G)\rightarrow{\rm LS}(G'),\hspace{1cm}\mathcal{E}_G(s)\mapsto\mathcal{E}_{G'}(\rho^\flat(s)).
$$

\begin{remark} In the following two cases, the map $\t_\rho$ is given by  a functor :

\noindent (1) If $L^\flat$ is a Levi factor of some parabolic subgroup of $G^\flat$ and if  $\rho^\flat:L^\flat\hookrightarrow G^\flat$ is the inclusion, then $\t_\rho=\t_L^G$.

\noindent (2) If $\rho^\flat:G^\flat\rightarrow G'{^\flat}$ is normal. Then  the map $\t_\rho$ is given by $(\rho^F)^*$ where $\rho:G'\rightarrow G$ is dual to $\rho^\flat$.

\label{remfunc}\end{remark}

\bigskip

\begin{remark} We can also define $\t_\rho$ from a morphism $\rho^\flat:N_{G^\flat}(T^\flat)\rightarrow G'{^\flat}$  as  a morphism $N_{G^\flat}(T^\flat)\rightarrow G'{^\flat}$  defines a map between the sets of  $F^\flat$-stable semisimple orbits of $G^\flat$ and $G'{^\flat}$.

Indeed let $s$ and $\sigma$ be two semisimple elements of $G^\flat$ that are $G^\flat$-conjugate. Then $s$ and $\sigma$ have $G^\flat$-conjugates $\overline{s}$ and $\overline{\sigma}$ respectively in $T^\flat$. The elements $\overline{s}$ and $\overline{\sigma}$ are $N_{G^\flat}(T^\flat)$-conjugate (indeed if $g\overline{s}g^{-1}=\overline{\sigma}$ then both $T^\flat$ and $gT^\flat g^{-1}$ are maximal tori of $C_{G^\flat}(\overline{\sigma})^o$ and so are conjugate by an element of $C_{G^\flat}(\overline{\sigma})^o$). Therefore the images of $\overline{s}$ and $\overline{\sigma}$ by $\rho^\flat$ are $G'{^\flat}$-conjugate. We thus have a well-defined map  between semisimple orbits. 
\end{remark}

\bigskip

Let $T^\flat$ be an $F^\flat$-stable maximal torus of $G^\flat$ and $T'{^\flat}$ be an $F^\flat$-stable maximal torus of $G'{^\flat}$ containing $\rho^\flat(T^\flat)$. Let $L'{^\flat}$ be the $F'{^\flat}$-stable Levi subgroup $C_{G'{^\flat}}(\rho^\flat(T^\flat))$ of $G'{^\flat}$, it contains $T'{^\flat}$. The morphism $\rho^\flat:T^\flat\rightarrow L'{^\flat}$ being normal, we get a dual morphism

$$
\rho:L'\rightarrow T.
$$

 By propositions \ref{remLus} and  \ref{func} we then have the following commutative diagram
 
 \begin{equation}
\xymatrix{{\rm LS}(T)\ar[rr]^{(\rho^F)^*}\ar[d]_{\t_{T}^G}&&{\rm LS}(L')\ar[d]^{\t_{L'}^{G'}}\\
{\rm LS}(G)\ar[rr]^{\t_\rho}&&{\rm LS}(G')}
\label{rhow}\end{equation}
These commutative diagrams, where $T^\flat$ runs over $F^\flat$-maximal tori of $G^\flat$,  characterize completely $\t_\rho$. In particular this shows that $\t_\rho$ does not depend on the choice of the  isomorphism $\overline{\F}_q^\times\simeq (\mathbb{Q}/\Z)_{p'}$ and the embedding $\overline{\F}_q^\times\hookrightarrow\Q^\times$.

\section{Gamma functions on finite reductive groups}

In this section $G$ is an arbitrary connected reductive group equipped with a geometric Frobenius $F:G\rightarrow G$. We will only be interested in gamma functions which are constant on Lusztig series.  We call them \emph{admissible}. A central function is then called admissible if the corresponding gamma function is admissible. The admissible central functions  on $G^F$  is the subspace generated by the functions of the form

$$
\sum_{\pi\in\mathcal{E}}\pi(1)\,\pi,
$$
with $\mathcal{E}$ a  Lusztig series of $(G,F)$.

\subsection{Gamma functions and normal morphisms}

Let $G'$ be another connected reductive group with Frobenius $F$ and  $f:G'\rightarrow G$ be a normal morphism which commutes with Frobenius. Recall (see Proposition \ref{func}) that the pullback functor $\pi\mapsto \pi\circ f^F$ induces a map $\t_f$ from the set of Lusztig series of $(G,F)$ to the set of Lusztig series of $(G',F)$.

We have the following lemma.

\begin{lemma} Let $\gamma^{G'}:\widehat{G'{^F}}\rightarrow\Q$ be admissible.   Let $\gamma^G:\widehat{G^F}\rightarrow\Q$ with corresponding central function $\phi^G$. Then the following assertions are equivalent : 

(1 ) $\gamma^G=\gamma^{G'}\circ \t_f$,

(2) $\phi^G=(f^F)_!(\phi^{G'})$.
\bigskip

\noindent If one of these two conditions is satisfied then $\gamma^G$ is also admissible.

\label{normalgamma}\end{lemma}

\begin{proof}Let us assume (2). For $\pi\in\widehat{G^F}$, we have

\begin{align*}
\gamma^G(\pi)&=\pi(1)^{-1}|G^F|\left(\phi^G,\pi^\vee\right)_{[G/G]^F}\\
&=\pi(1)^{-1}|G^F|\left(f_!(\phi^{G'}),\pi^\vee\right)_{[G/G]^F}\\
&=\pi(1)^{-1}|G'{^F}|\left(\phi^{G'},f^*\pi^\vee\right)_{[G'/G']^F}\\
&=\pi(1)^{-1}|G^F|\left(\phi^{G'},\sum_\chi a_\chi\,\chi^\vee\right)_{[G'/G']^F}
\end{align*}
where the sum is over the irreducible characters of $G'{^F}$.

We thus have
$$
\gamma^G(\pi)=\pi(1)^{-1}\sum_\chi a_\chi\, \chi(1)\gamma^{G'}(\chi).
$$
Since the $\chi$ such that $a_\chi\neq 0$ belongs to the Lusztig series $\t_f(\mathcal{E}_G(\pi))$ and since $\gamma^{G'}$ is constant on Lusztig series we deduce (1).
\end{proof}

\bigskip

\subsection{Gamma functions and Lusztig induction}\label{gammaLus}

We denote by $v_G$ the dimension of the unipotent radical of a Borel subgroup of $G$. If $H$ is another connected reductive group defined over $\F_q$, we put

$$
c_{H,G}:=q^{v_H-v_G}\epsilon_H\epsilon_G,
$$
where $\epsilon_G=(-1)^{\F_q-\text{rank}(G)}$.

We will use the following relations :

$$
c_{H,G}^{-1}=c_{G,H},\hspace{1cm} c_{G_1,G_2}\cdot c_{G_2,G_3}=c_{G_1,G_3}.
$$

Notice that $v_L-v_G=-{\rm dim}\, U_P$ for any $F$-stable Levi factor of a parabolic $P$ of $G$.
\bigskip

\noindent Let $L$ be an $F$-stable Levi factor of some parabolic subgroup $P$ of $G$. Recall (see Proposition \ref{remLus}) that the Lusztig induction functor $R_L^G:\mathcal{C}([L/L]^F)\rightarrow\mathcal{C}([G/G]^F)$ induces a map $\t_L^G$ from the set of Lusztig series of $(L,F)$ to the set of Lusztig series of $(G,F)$.

\begin{lemma}Let $\gamma^G:\widehat{G^F}\rightarrow\Q$ be admissible and $\gamma^L:\widehat{L^F}\rightarrow\Q$ with corresponding central function $\phi^L$. The following assertions are equivalent.

\noindent (i) $\gamma^L=c_{L,G}\, (\gamma^G\circ\t_L^G)$.

\noindent (ii) $\phi^L={^*}R^G_L(\phi^G)$.
\bigskip

If one of the two conditions hold then $\gamma^L$ is also admissible.

\label{delres}\end{lemma}

\begin{proof} The assertion (ii) is equivalent to :

$$
\left(\phi^L,\pi^\vee\right)_{[L/L]^F}=\left({^*}R_L^G(\phi^G),\pi^\vee\right)_{[L/L]^F}=\left(\phi^G,R_L^G(\pi^\vee)\right)_{[G/G]^F},
$$
for all $\pi\in\widehat{L^F}$.

For $\pi\in\widehat{L^F}$, write 

$$
R_L^G(\pi)=\sum_{\alpha\in\widehat{G^F}}n_\alpha\, \alpha,
$$
with $n_\alpha\in\Z$. Then 

\begin{align*}
|L^F|\,\left(\phi^G,R_L^G(\pi^\vee)\right)_{[G/G]^F}&=|L^F|\, \sum_\alpha n_\alpha\,\left(\phi^G,\alpha^\vee\right)_{[G/G]^F}\\
&=\frac{\gamma^G(\alpha)\,|L^F|}{|G^F|}\, R_L^G(\pi^\vee)(1)\\
&=\frac{\gamma^G(\t_L^G(\mathcal{E}_L(\pi)))\,|L^F|}{|G^F|}\, R_L^G(\pi^\vee)(1).
\end{align*}
The second equality holds for any $\alpha$ such that $n_\alpha\neq 0$ (because $\gamma^G$ is constant on Lusztig series by assumption).

But \cite[Proposition 10.2.9]{DM}

$$
R_L^G(\pi^\vee)(1)=\epsilon_G\epsilon_L|G^F/L^F|_{p'}\,\pi(1).
$$
Hence if (ii) holds then

$$
\gamma^L(\pi)=\gamma^G(\t_L^G(\mathcal{E}_L(\pi)))\epsilon_G\epsilon_L q^{-{\rm dim}\, U_P}
$$
hence (i). 

\end{proof}
\bigskip

Assume given, for any $F$-stable Levi factor $L$ of  a parabolic subgroup of $G$,  a function $\gamma^L:\widehat{L^F}\rightarrow\Q$, with corresponding central function $\phi^L$. Denote by ${\bf F}^L:\mathcal{C}(L^F)\rightarrow \mathcal{C}(L^F)^\iota$ the corresponding  isomorphism of $L^F\times L^F$-modules and  by ${\bf F}^{[L/L]}:\calC([L/L]^F)\rightarrow\calC([L/L]^F)$ the restriction of ${\bf F}^L$  to the subspace of central functions.

Let $\mathcal{T}^G$ be the collection of gamma functions $\gamma^T$ where $T$ describes the set of $F$-stable maximal tori of $G$. We say that $\mathcal{T}^G$ is admissible  if for any two geometrically conjugate DL pairs $(T,\theta)$ and $(T',\theta')$ of $G$ we have 

$$
\gamma^T(\theta)=c_{T,T'}\, \gamma^{T'}(\theta').
$$

\begin{proposition}The following assertions are then equivalent :

\noindent (1) $\mathcal{T}^G$ is admissible and for any inclusion $M\subset L$ of  $F$-stable Levi factors  we have
\begin{equation}
{\bf F}^{[L/L]}\circ R_M^L=c_{L,M}\, R_M^L\circ{\bf F}^{[M/M]}.
\label{commuteFour}\end{equation}

\noindent (2)  $\mathcal{T}^G$ is admissible  and for any $F$-stable Levi factor $L$ (of some parabolic subgroup of $G$) and any $F$-stable maximal torus $T$ of $L$, we have

\begin{equation}
\phi^L=\frac{1}{|W_L(T)|}\sum_{w\in W_L(T)}R_{T_w}^L(\phi^{T_w}).
\label{kernelfor}\end{equation}

\noindent (3) The function $\gamma^G$ is admissible and  for any inclusion of $F$-stable Levi factors  $M\subset L$
we have  $$\gamma^M=c_{M,L}\, \,\gamma^L\circ\t_M^L.$$

\noindent (4) The function $\phi^G$ is admissible and for any inclusion of $F$-stable Levi factors $M\subset L$, we have

$$
{^*}R_M^L(\phi^L)=\phi^M.
$$
\label{equiv}\end{proposition}

\begin{remark}Notice that if the equivalent conditions of the above proposition are satisfied then the gamma functions $\gamma^L$ are admissible.
\label{remLusRes}\end{remark}

\begin{proof}The equivalence of (3) with (4) follows from Lemma \ref{delres}. Let us now prove that (1) implies (2). Formula (\ref{kernelfor}) follows from the two formulas

\begin{equation}
\phi^L={\bf F}^L(1_1), \hspace{1cm} 1_1^{[L/L]}=\frac{1}{|W_L(T)|}\sum_{w\in W_L(T)}c_{T_w,L}\, R_{T_w}^L(1_1^{T_w}),
\label{11}\end{equation}
where $1_1^{[L/L]}$ and $1^{T_w}_1$ are the characteristic functions of $1$.

Let us  see that (2) implies (3). We may assume without loss of generalities that $M$ is a maximal torus $T$ of $L$.  Let $\theta\in\widehat{T^F}$ and let $\pi\in\mathcal{E}_L(T,\theta)$. Assume first that $\pi\in\widehat{L^F}$ is an irreducible constituent of $R_T^L(\theta)$. By Formula (\ref{inversephi})

$$
\gamma^L(\pi)=\frac{|L^F|}{\pi(1)}\left( \phi^L,\pi\right)_{[L/L]^F}.
$$
\begin{align*}
\left(\phi^L,\pi\right)_{[L/L]^F}&=\frac{1}{|W_L(T)|}\sum_{w\in W_L(T)}\left(R_{T_w}^L(\phi^{T_w}),\pi\right)_{[L/L]^F}\\
&=\frac{1}{|W_L(T)|}\sum_{w\in W_L(T)}\sum_{\alpha\in\widehat{T_w^F}}\left( \phi^{T_w},\alpha\right)_{T_w^F}\left( R_{T_w}^L(\alpha),\pi\right)_{[L/L]^F}.
\end{align*}
Since $\mathcal{T}^G$ is admissible we see that if $\left( R_{T_w}^L(\alpha),\pi\right)_{[L/L]^F}\neq 0$, then 

$$
\left( \phi^{T_w},\alpha\right)_{[T_w/T_w]^F}=|T_w^F|^{-1}\, \gamma^{T_w}(\alpha)=|T_w^F|^{-1}\, c_{T_w,T}\,\gamma^T(\theta).
$$
Hence

\begin{align*}
\left(\phi^L,\pi\right)_{[L/L]^F}&=\gamma^T(\theta)\frac{1}{|W_L(T)|}\sum_{w\in W_L(T)}|T_w^F|^{-1}c_{T_w,T}\,\sum_{\alpha\in\widehat{T_w^F}}\left( R_{T_w}^L(\alpha),\pi\right)_{[L/L]^F}\\
&=\gamma^T(\theta)\left( \frac{1}{|W_L(T)|}\sum_{w\in W_L(T)}c_{T_w,T}\,R_{T_w}^L(1_1),\pi\right)_{[L/L]^F}
\end{align*}
Using the second formula in (\ref{11}) we get

\begin{align*}
\left(\phi^L,\pi\right)_{[L/L]^F}&=c_{L,T}\, \gamma^T(\theta)\left(1_1,\pi\right)_{[L/L]^F}\\
&=c_{L,T}\, \gamma^T(\theta)\frac{\pi(1)}{|L^F|}.
\end{align*}

If  $\pi$ does not appear in $R_T^L(\theta)$ it will appear in some $R_{T'}^L(\theta')$ with $(T',\theta')$ in the geometric $L$-conjugacy class of $(T,\theta)$. Using the above calculation with $(T',\theta')$ instead of $(T,\theta)$ together with the  admissibility of $\mathcal{T}^G$, we get the required formula.

It remains to see that (3) implies (1). Let $\alpha$ be an irreducible character of $M^F$ and let

$$
R_M^L(\alpha)=\sum_{\pi\in\widehat{L^F}}n_\pi\,\pi
$$ be the decomposition into irreducible characters, then

$$
{\bf F}^L\left(R_M^L(\alpha)\right)=\sum_\pi\gamma^L(\pi)n_\pi\,\pi^\vee.
$$
Let $(T,\theta)$ be a DL pair such that $\alpha$ is an irreducible constituent of $R_T^M(\theta)$. Then any irreducible constituent of $R_M^L(\alpha)$  lives in $\mathcal{E}_G(T,\theta)$ by Proposition \ref{remLus} and so 

$$
{\bf F}^L\left(R_M^L(\alpha)\right)=c_{L,T}\,\gamma^T(\theta)\sum_\pi n_\pi\,\pi^\vee
$$
Since $R_L^G(\alpha^\vee)=R_L^G(\alpha)^\vee$ we have

$$
{\bf F}^L\left(R_M^L(\alpha)\right)=c_{L,T}\,  \gamma^T(\theta)R_M^L(\alpha^\vee)
$$
On the other hand, ${\bf F}^M(\alpha)=\gamma^M(\alpha)\alpha^\vee$ and 

$$
\gamma^M(\alpha)=c_{M,T}\, \gamma^T(\theta).
$$
Hence

$$
R_M^L\left({\bf F}^M(\alpha)\right)=c_{M,T}\, \gamma^T(\theta) R_M^L(\alpha^\vee).
$$
\end{proof}
\bigskip

\begin{remark}(1) If $\gamma^G$ is admissible and if we put 

$$
\gamma^L:=c_{L,G}\,\gamma^G\circ\t_L^G
$$
for any $F$-stable Levi subgroup $L$, then by transitivity of the $\t_L^G$, the family $\{\gamma^L\}_L$ satisfies the assertion (3) of the above proposition.
\label{rem}\end{remark}

\begin{proposition}When the center of $G$ is connected, Proposition \ref{equiv} remains true if we drop the assertion ``$\calT^G$ is admissible" in the fist assertion.

\label{equiv'}\end{proposition}

\begin{proof}Assume that the commutation formula (\ref{commuteFour}) holds for $L=G$ and $M$ a maximal torus $T$. For any irreducible character $\pi$ of $G^F$ which appears in $R_T^G(\theta)$ it is simple to see (following the lines of the proof of Proposition \ref{equiv}) that

$$
\gamma^G(\pi)=c_{G,T}\gamma^T(\theta).
$$
Therefore, if the commutation formula is true for $L=G$ and  $M$ any maximal torus of $G$, we deduce that 
$$
\gamma^T(\theta)=c_{T,T'}\,\gamma^{T'}(\theta')
$$
whenever $(T,\theta)\sim (T',\theta')$ where $\sim$ is the transitive closure of the relation defined by $(T,\theta)\sim (T',\theta')$ if $R_T^G(\theta)$ and $R_{T'}^G(\theta')$ contain a commun irreducible constituent. 

Assume that the center of $G$ is connected. We can then prove that the geometrically conjugacy classes of the DL pairs $(T,\theta)$ coincide with the equivalence classes for the above equivalent relation $\sim$ (see Remark \ref{ratLS}). We thus conclude that the commutation formula (\ref{commuteFour}) implies that $\calT^G$ is admissible.
\end{proof}
\bigskip

Fix now an $F$-stable maximal torus $T$ of $G$ with normalizer $N$. If the equivalent conditions of Proposition \ref{equiv} are satisfied then for any $\overline{w}\in H^1(F,N)$ we have 

$$
\phi^{T_w}={^*}R^G_{T_w}(\phi^G)
$$
and so by Remark \ref{rem-equi}, $\phi^{T_w}$ is $N_w^F$-invariant. Therefore, the collection of the $\phi^{T_w}$ where $\overline{w}$ runs over $H^1(F,N)$ defines a functions 

$$
\phi^{[T/N]}=\sum_{\overline{w}\in H^1(F,N)}\phi^{T_w}
$$
in $\calC([T/N]^F)$.

We thus have a well-defined operator ${\bf F}^{[T/N]}:\calC([T/N]^F)\rightarrow\calC([T/N]^F)$ 

$$
{\bf F}^{[T/N]}(f)=\pr_{2\, !}\left(\pr_1^*(f)\otimes \overline{m}^*\phi^{[T/N]}\right),
$$
for all $f\in\calC([T/N]^F)$, where

$$
\xymatrix{&[T/N]&\\
[T/N]&[(T\times T)/N]\ar[u]^{\overline{m}}\ar[r]^-{\pr_2}\ar[l]_-{\pr_1}&[T/N]}
$$
where $N$ acts diagonally on $T\times T$, and $\overline{m}$ is the quotient of the multiplication $m:T\times T\rightarrow T$.
\bigskip

\begin{remark}We have 
\begin{equation}
{\bf F}^{[T/N]}=\bigoplus_{\overline{w}\in H^1(F,N)}{\bf F}^{T_w},
\label{Fourdecomp}\end{equation}
where ${\bf F}^{T_w}:\calC([T_w^F/N_w^F])\rightarrow\calC([T_w^F/N_w^F])$ is defined from the kernel $\phi^{T_w}$.

\end{remark}

\bigskip

Denote by $\epsilon:\calC([T/N]^F)\rightarrow\calC([T/N]^F)$ the map $(f_w)_{\overline{w}\in H^1(F,N)}\mapsto(\epsilon_G\epsilon_{T_w}f_w)_{\overline{w}\in H^1(F,N)}$. 

Recall that 

$$
\epsilon_G\epsilon_{T_w}=(-1)^{\ell(w)}
$$
where $\ell(w)$ is the length of the image of $w\in N$ in $W=N/T$.

\begin{proposition}If the equivalent conditions of Proposition \ref{equiv} are satisfied then the diagram 
\begin{equation}
\xymatrix{\calC([T/N]^F)\ar[d]_{{\bf F}^{[T/N]}}\ar[rr]^{(\I_{[T/N]}^G\circ\epsilon)(v_G)}&&\calC([G/G]^F)\ar[d]^{{\bf F}^{[G/G]}}\\
\calC([T/N]^F)\ar[rr]^{\I_{[T/N]}^G}&&\calC([G/G]^F),}
\label{comdiag1}\end{equation}
commutes, and 

\begin{equation}
\I_{[T/N]}^G(\phi^{[T/N]})=\phi^G.
\label{mainfor1}\end{equation}
\label{propmain0}\end{proposition}

\begin{proof} The diagram commutes thanks to the commutation formula (\ref{commuteFour}) and the decomposition (\ref{Fourdecomp}).

The equality (\ref{mainfor1}) follows from Formula (\ref{kernelfor}). 
\end{proof}

\begin{remark} Formula (\ref{mainfor1}) is a consequence of the commutativity of Diagram (\ref{comdiag1}). Indeed

$$
{\bf F}^{[G/G]}(1_1^{[G/G]})=\phi^G,\hspace{1cm}{\bf F}^{[T/N]}(1_1^{[T/N]})=\phi^{[T/N]}
$$
where $1_1^{[G/G]}$ denote the characteristic function of the neutral element $1$ and $1_1^{[T/N]}:=\sum_{\overline{w}\in H^1(F,N)}1^{T_w}_1$. The formula follows thus from the commutativity of Diagram (\ref{comdiag1}) and the formula (see Formula (\ref{11}))

\begin{align*}
(\I_{[T/N]}^G\circ\epsilon)(v_G)(1_1^{[T/N]})&=\frac{1}{|W|}\sum_{w\in W}\epsilon_G\epsilon_{T_w}q^{-v_G}R_{T_w}^G(1^{T_w}_1)\\
&=1_1^{[G/G]}.
\end{align*}
\label{remcom}\end{remark}

\section{Exotic Fourier operator}

\subsection{Standard Fourier operators}\label{Fourierstandard}

Let $G'$ be a connected reductive group with geometric Frobenius $F:G'\rightarrow G'$. We assume that the pair $(G',F)$ is \emph{standard}, namely that it is isomorphic to a pair of the form $((\GL_{n_1})^{m_1}\times\cdots\times(\GL_{n_r})^{m_r},F)$, with $n_1>n_2>\cdots>n_r$ and  where $F$ is the composition of the standard Frobenius with an element $\sigma$ of $S_{m_1}\times\cdots\times S_{m_r}$. In other words $(G',F)$ is isomorphic to a rational Levi factor of some parabolic subgroup  of $\GL_n$, with $n=\sum_i n_im_i$, and where  the $\F_q$-structure on $\GL_n$ is standard. In particular, if $L'$ is an $F$-stable Levi factor of some parabolic subgroup of $G'$, then the pair $(L',F)$ is also standard.

For each $i=1,\dots, r$,  let $\lambda_i=(\lambda_{i,1},\lambda_{i,2},\dots,\lambda_{i,{s_i}})$ be the partition of $m_i$ given by the decomposition of the $i$-th coordinate of $\sigma$ into disjoint cycles. We have

$$
G'{^F}\simeq\prod_{i=1}^r\prod_{j=1}^{s_i}\GL_{n_i}(\F_{q^{\lambda_{i,j}}}).
$$
We denote by $F$ the corresponding Frobenius on $\frakg':=(\gl_{n_1})^{m_1}\times\cdots\times(\gl_{n_r})^{m_r}$, then

$$
\frakg'{^F}\simeq\prod_{i=1}^r\prod_{j=1}^{s_i}\gl_{n_i}(\F_{q^{\lambda_{i,j}}}).
$$

We consider on any algebra of the form  $\gl_{s_1}\times\cdots\times\gl_{s_r}$ the trace form $\Tr(x)=\sum_{i=1}^r\Tr(x_i)$. It is compatible with restriction to Levi subalgebras and commutes with Frobenius. The algebra $\frakg'$ is therefore equipped with a trace form $\Tr$ which commutes with Frobenius and which is compatible with restriction to Levi subalgebras.
\bigskip

We fix a non-trivial additive character $\psi:\F_q\rightarrow\Q^\times$. We define the standard  Fourier transform ${\bf F}^{\frakg'}:\mathcal{C}(\frakg'{^F})\rightarrow\mathcal{C}(\frakg'{^F})$ by the formula

$$
{\bf F}^{\frakg'}(f)(x)=\sum_{y\in \frakg'{^F}}\psi(\Tr(xy))f(y),
$$
for all $f\in\mathcal{C}(\frakg'{^F})$ and $x\in \frakg'{^F}$.

Given an $F$-stable Levi subgroup $L'$ of $G'$ with corresponding Lie algebra $\frakl'$, one can extend Lusztig induction $R_{L'}^{G'}:\calC([L'/L']^F)\rightarrow\calC([G'/G']^F)$ using the embedding $j:[G'/G']\rightarrow[\frakg'/G']$ to an induction $R_{\frakl'}^{\frakg'}:\calC([\frakl'/L']^F)\rightarrow\calC([\frakg'/G']^F)$ (see \cite[Chapter 3]{letellier}) such that the Lusztig inductions commute with $j_!$ and $j^*$, i.e.

$$
R_{\frakl'}^{\frakg'}\circ j_!=j_!\circ R_{L'}^{G'},\hspace{1cm}j^*\circ R_{\frakl'}^{\frakg'}=R_{L'}^{G'}\circ j^*.
$$
We have the following commutation formula between Lusztig induction and Fourier transforms \cite[Corollary 6.2.17]{letellier}.

\begin{theorem}
$$
{\bf F}^{[\frakg'/G']}\circ R_{\frakl'}^{\frakg'}=c_{G',L'}\, R_{\frakl'}^{\frakg'}\circ{\bf F}^{[\frakl'/L']},
$$
where ${\bf F}^{\frakl'}$ is the standard Fourier transform, i.e. given by the kernel $\psi\circ\Tr$ on $\frakl'{^F}$.
\label{let}\end{theorem}
\bigskip

\begin{remark}As $(G',F)$ is a standard,  the only proper Levi subgroups of $G'$ which support a \emph{cuspidal pair} are the maximal tori, and so by \cite[Corollary 6.2.6]{letellier}, the proof of Theorem \ref{let} reduces to the case where the Levi $L'$ is a maximal torus of $G'$. Fix an $F$-stable maximal torus $T'$ of $G'$ with normalizer $N'$. For $w\in N'$ denotes by $\frakt'_w$ the Lie algebra of $T'_w$. 

As the stack $[T'/N']$ takes care of all $F$-stable maximal tori of $G'$ (up to rational conjugacy), the above theorem is equivalent to the commutativity of the following diagram

\begin{equation}
\xymatrix{\calC([\frakt'/N']^F)\ar[d]_{{\bf F}^{[\frakt'/N']}}\ar[rr]^{(\I_{[\frakt'/N']}^{\frakg'}\circ\epsilon)(v_{G'})}&&\calC([\frakg'/G']^F)\ar[d]^{{\bf F}^{[\frakg'/G']}}\\
\calC([\frakt'/N']^F)\ar[rr]^{\I_{[\frakt'/N']}^{\frakg'}}&&\calC([\frakg'/G']^F),}
\label{comdiagLie}\end{equation}
where 

$$
\I_{[\frakt'/N']}^{\frakg'}(f)=\sum_{\overline{w}\in H^1(F,N')}\frac{1}{|W_w^F|}R_{\frakt'_w}^{\frakg'}(f_w)
$$
for $f=(f_w)_{\overline{w}\in H^1(F,N')}\in\calC([\frakt'/N']^F)$, and 

$$
{\bf F}^{[\frakt'/N']}=\bigoplus_{\overline{w}\in H^1(F,N')}{\bf F}^{\frakt'_w},
$$
with ${\bf F}^{\frakt'_w}:\calC([\frakt_w'{^F}/N'_w{^F}])\rightarrow\calC([\frakt_w'{^F}/N'_w{^F}])$ defined from the kernel $\psi\circ\Tr$ on $\frakt'_w{^F}$.

\end{remark}

\bigskip

We define the standard Fourier operator ${\bf F}^{G'}:\mathcal{C}(G'{^F})\rightarrow\mathcal{C}(G'{^F})$ as 

$$
{\bf F}^{G'}:=j^*\circ{\bf F}^{\frakg'}\circ j_!,
$$
where $j:G'\hookrightarrow\frakg'$ is the inclusion. The associated kernel $\phi^{G'}$ is the function $\psi\circ\Tr|_{G'}$ on $G'{^F}$.

The following result is a consequence of Theorem \ref{let}.

\begin{corollary}For any $F$-stable Levi factor $L'$ of some parabolic subgroup of $G'$, we have

$$
{\bf F}^{[G'/G']}\circ R_{L'}^{G'}=c_{G',L'}\, R_{L'}^{G'}\circ{\bf F}^{[L'/L']}.
$$
\label{commuteres}\end{corollary}

Let $\gamma^{L'}$ be the gamma function associated to ${\bf F}^{L'}$. Since the center of $G'$ is connected, by Proposition \ref{equiv'} we have the following result.

\begin{corollary} The family $\{\gamma^{L'}\}_{L'}$ satisfies the equivalent conditions of Proposition \ref{equiv}. In particular, the functions $\gamma^{L'}$ are admissible and

\begin{equation}
{^*}R_{L'}^{G'}(\psi\circ\Tr|_{G'})={^*}R_{L'}^{G'}(\phi^{G'})=\phi^{L'}=\psi\circ\Tr|_{L'}.
\label{equares}\end{equation}

\end{corollary}

Moreover the gamma function $\gamma^{L'}$ never vanish as for any DL pair $(T',\theta')$ of $G'$, the value $\gamma^{T'}(\theta')$ is a Gauss sum and therefore is non-zero.

\begin{proposition}Let $S$ be a torus and assume given a morphism $\rho:G'\rightarrow S$ defined over $\F_q$. Then if $T'$ is an $F$-stable maximal tori of $G'$ we have

$$
(\rho^F)_!\left(\psi\circ\Tr|_{G'}\right)=c_{G',T'}\, (\rho_{T'}^F)_!\left(\psi\circ\Tr|_{T'}\right)
$$
where $\rho_{T'}:T'\rightarrow S$ is the restriction of $\rho$ to $T'$.
\label{S}\end{proposition}

\begin{proof}  Let us give here a simple proof using gamma functions and (\ref{equares}). For a more direct approach see Appendix B.

The morphism $\rho$ being normal is in duality with a morphism 

$$
\rho^\flat:S^\flat\rightarrow G'{^\flat}=G'.
$$
By Remark \ref{remfunc} we thus have the following commutative triangle

$$
\xymatrix{{\rm LS}(S)\ar[rr]^{\rho_{T'}^*}\ar[rrd]_{\rho^*}&&{\rm LS}(T')\ar[d]^{\t_{T'}^{G'}}\\
&&{\rm LS}(G')}
$$
Therefore 

$$
\gamma^S_\rho:=\gamma^{G'}\circ (\rho^F)^*=\gamma^{G'}\circ \t_{T'}^{G'}\circ (\rho_{T'}^F)^*.
$$
By (\ref{equares}) and Lemma \ref{delres}, we have $\gamma^{T'}=c_{T',G'}\,\gamma^{G'}\circ\t_{T'}^{G'}$. Therefore 

$$
\gamma^S_\rho=c_{G',T'}\,\gamma^{T'}\circ(\rho_{T'}^F)^*
$$
and so we conclude by Lemma \ref{normalgamma} as the kernel corresponding to the gamma function $\gamma^S_\rho$ is $(\rho^F)_!\phi^{G'}$ and the kernel corresponding to the gamma function $\gamma^{T'}\circ (\rho_{T'}^F)^*$ is $(\rho_{T'}^F)_!\phi^{T'}$.

\end{proof}

\subsection{Spectral definition of exotic Fourier operators}\label{Spectraldef}

Fix a standard pair $(G',F)$ as in \S \ref{Fourierstandard} and let $\rho^\flat:G^\flat\rightarrow G'{^\flat}=G'$ be defined over $\F_q$. 

Consider the gamma function on $\widehat{G^F}$

$$
\gamma^G_\rho:=c_{G,G'}\gamma^{G'}\circ\t_\rho.
$$
where $\gamma^{G'}$ is as in \S \ref{Fourierstandard} and denote by 

$$
{\bf F}^G_\rho:\calC(G^F)\rightarrow\calC(G^F)^\iota
$$
the corresponding operator that we call \emph{exotic Fourier operator}. They were first considered by Bravermann and Kazhdan \cite{BK}.

For any $F$-stable maximal torus $T$ of $G$ put

$$
\gamma^T_\rho:=c_{T,G}\,\gamma^G_\rho\circ \t_T^G.
$$
The kernel $\phi^T_\rho$ corresponding to $\gamma^T_\rho$ can be explicitly computed as follows. 

\begin{lemma}If $T'$ is an $F$-stable maximal torus of $G'$ that contains the image $\rho^\flat(T^\flat)$ then 

$$
\phi^T_\rho=c_{T,T'}\, (\rho_{T'}^F)_!(\phi^{T'})
$$
where $\rho_{T'}:T'\rightarrow T$ is dual to the restriction $T^\flat\rightarrow T'$ of $\rho^\flat$.
\label{remcru}
\end{lemma}

\begin{proof}The proof is completely similar to that of Proposition \ref{S}. We have the commutative diagram

$$
\xymatrix{{\rm LS}(T)\ar[d]_{\t_T^G}\ar[rr]^{\rho_{T'}^*}&&{\rm LS}(T')\ar[d]^{\t_{T'}^{G'}}\\
{\rm LS}(G)\ar[rr]^{\t_\rho}&&{\rm LS}(G')}
$$
from which we deduce

\begin{align*}
\gamma_\rho^T&=c_{T,G'}\,\gamma^{G'}\circ\t_\rho\circ\t_L^G\\
&=c_{T,G'}\,\gamma^{G'}\circ\t_{T'}^{G'}\circ\rho_{T'}^*\\
&=c_{T,T'}\, \gamma^{T'}\circ\rho_{T'}^*.
\end{align*}
The last equality follows from Lemma \ref{delres} as ${^*}R^{G'}_{T'}(\phi^{G'})=\phi^{T'}$ by Formula (\ref{equares}).

The lemma follows thus from Lemma \ref{normalgamma}.

\end{proof}

\begin{remark}The above lemma implies that if $T''$ is another  $F$-stable maximal torus of $G'$ containing $\rho^\flat(T^\flat)$ then 

$$
\epsilon_{T'}\,(\rho_{T'}^F)_!\left(\psi\circ\Tr|_{T'}\right)=\epsilon_{T''}\,(\rho_{T''}^F)_!\left(\psi\circ\Tr|_{T''}\right).
$$
This can be also deduced from Proposition \ref{S}. To see that we consider the Levi subgroup  $L':=C_{G'}(\rho^\flat(T^\flat))$ of $G'$. The restriction $T^\flat\rightarrow L'$ of $\rho^\flat$ is then normal and we apply the proposition to the dual morphism $\rho_{L'}:L'\rightarrow T$. Notice that

$$
\phi^T_\rho=c_{T,L'}\, (\rho_{L'}^F)_!(\psi\circ\Tr|_{L'})
$$
\end{remark}
\bigskip

We now fix a maximally split $F$-stable maximal torus $T$ of $G$ with normalizer $N$.
\bigskip

From the definition of  $\gamma^T_\rho$ and Remark \ref{rem}, the conditions of Proposition \ref{equiv} are satisfied and so by Proposition \ref{propmain0} we have the following result.

\begin{theorem}The diagram 

\begin{equation}
\xymatrix{\calC([T/N]^F)\ar[d]_{{\bf F}_\rho^{[T/N]}}\ar[rr]^{(\I_{[T/N]}^G\circ\epsilon)(v_G)}&&\calC([G/G]^F)\ar[d]^{{\bf F}_\rho^{[G/G]}}\\
\calC([T/N]^F)\ar[rr]^{\I_{[T/N]}^G}&&\calC([G/G]^F),}
\label{comdiag}\end{equation}
commutes, and 

\begin{equation}
\I_{[T/N]}^G(\phi_\rho^{[T/N]})=\phi_\rho^G.
\label{mainfor}\end{equation}
\label{propmain}\end{theorem}

\begin{remark}Notice that  in the above diagram, ${\bf F}_\rho^{[G/G]}$ is defined spectrally (using gamma functions) while the kernel of ${\bf F}^{[T/N]}_\rho$ can be explicitly defined in geometrical terms by Remark \ref{remcru}. The equality (\ref{mainfor}) gives then an explicit formula for $\phi^G_\rho$.
\end{remark}

\section{Geometric realizations}

The stacks we consider are algebraic $\overline{\F}_q$-stacks of finite type. For such a  stack $\calX$ we denote by $\mathcal{D}_c^b(\calX)$ the bounded ``derived category" of complexes of $\Q$-sheaves on $\calX$ with constructible cohomology.  

\subsection{Preliminaries}\label{Prelim}

Let $H$ is a finite group acting on the right on an $\overline{\F}_q$-scheme $X$. An $H$-equivariant complex on $X$ is a pair  $(K,\theta)$ with $K\in \dcb(X)$ and $\theta=(\theta_h)_{h\in H}$ a collection of isomorphisms

$$
\theta_h:h^*(K)\simeq K
$$
satisfying the two conditions :

$\theta_1={\rm Id}$,

$\theta_{uv}=\theta_u\circ u^*(\theta_v)$ for all $u,v\in H$.
\bigskip

We denote by $\dcb(X;H)$ the category whose objects are the  $H$-equivariant complexes and the morphisms $(K,\theta)\rightarrow (K',\theta')$ is  ${\rm Hom}(K,K')^H$.
\bigskip

If $H$ acts trivially on $X$, then an $H$-equivariant complex is a pair $(K,\theta)$ with $\theta:H\rightarrow{\rm Aut}(K)$ a group homomorphism, in which case we say that $H$ acts on $K$.
\bigskip

\begin{remark}We can extend the above definition to  $H$-equivariant complexes on a stack $\calX$.  If we do not want to bother with action of finite group on stacks, we can proceed as in \cite[\S 2.9]{LL}. In this paper we will not need to work in this more general context in a crucial way. Indeed, although we will mention $W$-equivariant perverse sheaves on the stack $[T/T]$ (for aesthetic reasons), we notice that the category of perverse sheaves on $[T/T]$ is equivalent to that of perverse sheaves on $T$.
\end{remark}
\bigskip

Recall that for a stack $\calX$, the category $\mathcal{D}_c^b(\calX)$ is a triangulated category with bounded $t$-structure is Krull-Remak-Schmidt, namely each object decomposes into a finite direct sum of indecomposable objects. Recall that in a Krull-Remak-Schmidt category,  all idempotent of endomorphism rings splits. 

In particular if $H$ is a finite group acting on $K\in\mathcal{D}_c^b(\calX)$ then we have an isomorphism

$$
K\simeq\bigoplus_{\chi\in \widehat{H}}K_\chi,
$$
where $K_\chi$ is the kernel of the idempotent $1-e_\chi\in{\rm End}(K)$ with

$$
e_\chi:=\frac{\chi(1)}{|H|}\sum_{h\in H}\overline{\chi(h)}\, \theta_h.
$$
We will denote by $K^H$ the $H$-invariant part $K_1$ of $K$.
\bigskip

If $\pi:X\rightarrow[X/H]$ is the quotient map, then for any $K\in\dcb([X/H])$, there exists a canonical $H$-equivariant structure on $\pi^*(K)$ and in this way $\pi^*$ realizes an equivalence of categories between $\dcb([X/H])$ and the category of $H$-equivariant complexes on $X$. The inverse functor $\dcb(X;H)\rightarrow\dcb([X/H])$ is $(K,\theta)\mapsto \pi_!(K,\theta)^H$.

\bigskip

Assume now that $\calX$ is a stack defined over $\F_q$. Denote by $F:\calX\rightarrow\calX$ the corresponding geometric Frobenius and denote by $\calX^F$ the groupo\"id of $\F_q$-points.  An  $F$-equivariant complex is a pair $(K,\varphi)$ with $K\in \mathcal{D}_c^b(\calX)$ and $\varphi:F^*K\simeq K$. If $(K,\varphi)$ and $(K',\varphi')$ are two $F$-equivariant complexes, the Frobenius $F$ acts on ${\rm Hom}(K,K')$ as $f\mapsto \varphi'\circ F^*( f)\circ\varphi^{-1}$. We let $\dcb(\calX;F)$ be the category whose objects are $F$-equivariant complexes and the set morphisms $(K,\varphi)\rightarrow (K',\varphi')$ is ${\rm Hom}(K,K')^F$. We denote by $\bfX:\mathcal{D}_c^b(\calX;F)\rightarrow\mathcal{C}(\calX^F)$ the map which sends $(K,\varphi)\in \mathcal{D}_c^b(\calX;F)$ to its characteristic function  $\bfX_{K,\varphi}:x\mapsto\sum_i(-1)^i\Tr(\varphi_x^i,\,\mathcal{H}_x^iK)$.

\begin{theorem}If $f:\calX\rightarrow \calY$ commutes with Frobenius endomorphisms, the following diagrams commute

$$
\xymatrix{\mathcal{D}_c^b(\calX;F)\ar[rr]^{f_!}\ar[d]_\bfX&&\mathcal{D}_c^b(\calY;F)\ar[d]^\bfX\\
\mathcal{C}(\calX^F)\ar[rr]^{(f^F)_!}&&\mathcal{C}(\calY^F)} \hspace{2cm}  \xymatrix{\mathcal{D}_c^b(\calY;F)\ar[rr]^{f^*}\ar[d]_\bfX&&\mathcal{D}_c^b(\calX;F)\ar[d]^\bfX\\
\mathcal{C}(\calY^F)\ar[rr]^{(f^F)^*}&&\mathcal{C}(\calX^F)}
$$
where $f^F:\calX^F\rightarrow \calY^F$.
\end{theorem}
\bigskip

Assume that $X$ is defined dover $\F_q$ with associated geometric Frobenius $F:X\rightarrow X$.  We also fix an element of ${\rm Aut}(H)$ which will be thought as a Frobenius on $H$ and which will be denoted also by $F$ (such an automorphism provides a finite \'etale group  $H_o$ over $\F_q$ such that $H=H_o\times_{\F_q}\overline{\F}_q$).

Assume now that the right action of $H$ on $X$ is compatible with Frobenius $F$, i.e., 

$$
F(x\cdot h)=F(x)\cdot F(h)
$$
for all $x\in X$ and $h\in H$.
\bigskip

We denote by $\dcb(X;H,F)$ the category whose objects are triples $(K,\theta,\varphi)$ with $K\in\dcb(X)$, $\theta$ an $H$-equivariant structure on $K$ and $\varphi:F^*K\simeq K$ such that the following diagram commutes for all $h\in H$:

\begin{equation}
\xymatrix{h^*F^*(K)\ar[d]_{h^*(\varphi)}\ar[rr]^{F^*(\theta_{F(h)})}&&F^*K\ar[d]^\varphi\\
h^*K\ar[rr]^{\theta_h}&&K}
\label{comp}\end{equation}
We see this compatibility condition as a cocycle condition  in $H\rtimes \langle F^{-1}\rangle$ which we let act on $X$ on the right by 

$$
x\cdot (hF^{-1})=F(x\cdot h)
$$
for $x\in X$ and $h\in H$. 

\noindent Indeed we have  $(hF^{-1})^*K=h^*F^*(K)$ and if we put 

(1) $\theta_{F^{-1}}:=\varphi$, 

(2) $\theta_{hF^{-1}}:=\theta_h\circ h^*(\theta_{F^{-1}}): (hF^{-1})^*K\rightarrow K$ for $h\in H$,

\noindent then, since $hF^{-1}=h\cdot F^{-1}=F^{-1}\cdot F(h)$ in $H\rtimes\langle F^{-1}\rangle$, we want to have

$$
\theta_{hF^{-1}}=\theta_{F^{-1}}\circ F^*(\theta_{F(h)}),
$$
i.e. we want the commutativity of the diagram (\ref{comp}).
\bigskip

A morphism $(K,\theta,\varphi)\rightarrow (K',\theta',\varphi')$ in $\dcb(X;H,F)$ is a morphism $f:K\rightarrow K'$ compatible with the $H$-equivariant and $F$-equivariant structures and such that the following diagram commutes for all $h\in H$

$$
\xymatrix{(hF^{-1})^*K\ar[d]_{\theta_{hF^{-1}}}\ar[rr]^{(hF^{-1})^*(f)}&&(hF^{-1})^*K'\ar[d]^{\theta'_{hF^{-1}}}\\
K\ar[rr]^f&&K'}
$$

Note that if $f:X\rightarrow Y$ is an $H$-equivariant morphism which commutes with Frobenius, then $f^*$ and $f_!$ induces functors $f_!:\dcb(X;H,F)\rightarrow\dcb(Y;H,F)$ and $f^*:\dcb(Y;H,F)\rightarrow\dcb(X;H,F)$.
\bigskip

\begin{remark}Assume that $H$ acts trivially on $X$ and let $(K,\theta,\varphi)\in\dcb(X;H,F)$. From the diagram (\ref{comp}), we have the commutative diagram

$$
\xymatrix{(F^*K)_{\chi\circ F^{-1}}\ar[rr]\ar[d]&&F^*K\ar[d]_{\varphi}\ar[rr]^{1-f_{\chi\circ F^{-1}}}&&F^*K\ar[d]^{\varphi}\\
K_\chi\ar[rr]&&K\ar[rr]^{1-e_\chi}&&K}
$$
for any $\chi\in\widehat{H}$ where 

$$
e_\chi=\frac{\chi(1)}{|H|}\sum_{h\in H}\overline{\chi(h)}\theta_h,\hspace{1cm}f_\chi=\frac{\chi(1)}{|H|}\sum_{h\in H}\overline{\chi(h)}(F^*\theta)_h.
$$
On the other hand

$$
F^*(K_\chi)\simeq (F^*K)_\chi
$$
as by definition $(F^*\theta)_h=F^*(\theta_h)$. Therefore when $\chi=\chi\circ F^{-1}$, the Weil structure $\varphi$ restricts to a Weil structure $\varphi_\chi:F^*(K_\chi)\simeq K_\chi$ and

$$
{\bf X}_{K,\varphi}=\sum_{\chi=\chi\circ F^{-1}}{\bf X}_{K_\chi,\varphi_\chi}.
$$
In fact, for all $w\in H$, we have

$$
{\bf X}_{K,\theta_w\circ\varphi}=\sum_{\chi=\chi\circ F^{-1}}\chi(w)\,{\bf X}_{K_\chi,\varphi_\chi}.
$$

When $\chi=1$, we put $(K^H,\varphi^H):=(K_1,\varphi_1)$. Then

\begin{equation}
{\bf X}_{K^H,\varphi^H}=\frac{1}{|W|}\sum_{w\in W}{\bf X}_{K,\theta_w\circ\varphi}.
\label{invariant}\end{equation}
\end{remark}
\bigskip

\begin{proposition}Consider the quotient map $\pi:X\rightarrow[X/H]$. Then $\pi^*$ realizes an equivalence of categories
$\dcb([X/H],F)\rightarrow\dcb(X;H, F)$.
\label{carquo}\end{proposition}

\begin{proof}The functor $(K,\theta,\varphi)\mapsto ((\pi_!K)^H,(\pi_!\varphi)^H)$ is the inverse functor of $\pi^*$.
\end{proof}

Finally note  that, for any $h\in H$, we have a functor

$$
\mathfrak{o}_h:\dcb(X;H,F)\rightarrow\dcb(X;F\circ h),\qquad (K,\theta,\varphi)\mapsto (K, \theta_h\circ h^*(\varphi))
$$
where $F\circ h:X\rightarrow X, x\mapsto F(x\cdot h)$.

\begin{remark}Let $(K,\theta,\varphi)\in\dcb(X;H,F)$ and let $(\overline{K},\overline{\varphi})$ be the corresponding object in $\dcb([X/H];F)$. Then, for $\overline{h}\in H^1(F,H)$, the $\overline{h}$-coordinate of ${\bf X}_{\overline{K},\overline{\varphi}}$ in the decomposition
$$
\calC([X/H]^F)=\bigoplus_{\overline{h}\in H^1(F,H)}\calC(X^{F\circ h})^{H^{F\circ h}}
$$
is the characteristic function of $\mathfrak{o}_h(K,\theta,\varphi)$.

\label{remdecomp}\end{remark}

\bigskip

We will regard the constant sheaf $\Q$ on $\calX$ as an object of $\mathcal{D}_c^b(\calX)$ concentrated in degree $0$ and if $\calX$ is irreducible we denote by $\IC_\calX$ the intersection cohomology complex on $\calX$ (i.e. the intermediate extension of the smooth $\ell$-adic sheaf $\Q$ on some smooth non-empty open substack of $\calX$).

We denote respectively by $\calM(\calX)$ (resp. $\calM(\calX;F)$, $\calM(\calX;H,F)$) the subcategory of $\dcb(\calX)$ (resp. $\dcb(\calX;F)$, $\dcb(\calX;H,F)$) of perverse sheaves on $\calX$ for the auto-dual perversity.

\subsection{Geometric induction and Deligne-Lusztig induction}\label{ind}

Let $T$ be an $F$-stable maximal torus of $G$ with normalizer $N$ and Weyl group $W=N/T$. Let $B$ be a Borel subgroup containing $T$.

Put 

$$
\car_T:=T/\!/W
$$
and consider the quotient stacks $[T/T]$, $[B/B]$ and $[G/G]$ with respect to the conjugation action. 

Denote by $\B(T)$ the classifying stack of $T$-torsors. Since $T$ acts trivially on itself, we have

$$
[T/T]\simeq T\times\B(T).
$$
From \cite[\S 2.9, \S 7.2]{LL} we can define  a pair of adjoint functors $({^*}\calI^G_{[T/N]},\calI^G_{[T/N]})$

$$
\xymatrix{\calM([T/N])\ar@/_/[rr]_-{\calI_{[T/N]}^G}&&\calM([G/G])\ar@/_/[ll]_-{{^*}\calI^G_{[T/N]}}}.
$$
Let us explain their construction. 

Consider the commutative diagram

\begin{equation}
\xymatrix{[T/T]\ar[d]_s\ar[d]&[B/B]\ar[l]_q\ar[rd]^p\ar[d]^{(q',p)}\ar[ld]_{q'}&\\
T&T\times_{\car_T}[G/G]\ar[r]_-{\pr_2}\ar[l]^-{\pr_1}&[G/G]}
\label{pInd}\end{equation}
where $s:[T/T]\rightarrow T$ is the projection.

By \cite{BY}, Lusztig functor $\Ind_{[T/T]}^G:\dcb([T/T])\rightarrow\dcb([G/G]),\, K\mapsto p_*q^!(K)$ preserves perverse sheaves. Moreover the functor $s^![{\rm dim}\, T]({\rm dim}\, T):\calM(T)\rightarrow\calM([T/T])$ is an equivalence of categories with inverse functor ${^p}\calH^0\circ \left(s_![-{\rm dim}\, T](-{\rm dim}\, T)\right)$.

From diagram (\ref{pInd}), the functor $\Ind_{[T/T]}^G:\calM([T/T])\rightarrow\calM([G/G])$ decomposes as 

$$
\Ind_{[T/T]}^G=\Ind_T^G\circ {^p}\calH^0\circ \left(s_![-{\rm dim}\, T](-{\rm dim}\, T)\right)
$$
where $\Ind_T^G:\calM(T)\rightarrow\calM([G/G]),\, K\mapsto p_*q'{^!}(K)[{\rm dim}\, T]({\rm dim}\, T)$.
\bigskip

\begin{remark}Since $(q',p)_!\Q=\IC_{T\times_{\car_T}[G/G]}$, the functor $\Ind_T^G:\calM(T)\rightarrow\calM([G/G])$ can be defined from the bottom correspondence of (\ref{pInd}) with kernel $\IC_{T\times_{\car_T}[G/G]}$ (see \cite[Lemma 2.15]{LL}), i.e.

$$
\Ind_T^G(K)=\pr_{2*}\underline{\rm Hom}\left(\IC_{T\times_{\car_T}[G/G]},\pr_1^!K\right)[{\rm dim}\, T]({\rm dim}\, T)
$$
Since $\IC_{T\times_{\car_T}[G/G]}$ is naturally $F$-equivariant and $W$-equivariant, the functor $\Ind_T^G$ preserves $F$-equivariance and for any $K\in\calM(T;W)$ we get a natural action of $W$ on $\Ind_T^G(K)$. The same is thus true for $\Ind_{[T/T]}^G$, namely we get a functor

$$
\Ind_{[T/T]}^G:\calM([T/T];W,F)\simeq\calM(T;W,F)\rightarrow\calM([G/G];W,F)
$$
where $W$ acts trivially on $[G/G]$ (I.e. $W$-equivariant sheaves on $[G/G]$ are sheaves equipped with an action of $W$).
\end{remark}

Consider the cartesian diagram

$$
\xymatrix{[T/T]\ar[rr]^s\ar[d]_{\pi_{[T/T]}}&&T\ar[d]^{\pi_T}\\
[T/N]\ar[rr]^{\overline{s}}&&[T/W]}
$$
We consider the functor $\calI_{[T/W]}^G:\calM([T/W];F)\rightarrow\calM([G/G];F)$ defined from the cohomological correspondence

\begin{equation}
\xymatrix{[T/W]&&[T/W]\times_{\car_T}[G/G]\ar[rr]^-{\pr_2}\ar[ll]_-{\pr_1}&&[G/G]}
\label{diagT/W}\end{equation}
with kernel $\IC_{[T/W]\times_{\car_T}[G/G]}$. We have \cite[Proposition 2.21]{LL}

\begin{equation}
\Ind_T^G=\calI_{[T/W]}^G\circ\pi_{T\, !}
\label{fact0}\end{equation}
and so we end up with a factorization

\begin{equation}
\Ind_{[T/T]}^G=\calI_{[T/N]}^G\circ\pi_{[T/T]\, !}
\label{fact}\end{equation}
where 

$$
\calI_{[T/N]}^G:=\calI_{[T/W]}^G\circ{^p}\calH^0\circ \left(\overline{s}_![-{\rm dim}\, T](-{\rm dim}\, T)\right).
$$

\begin{remark}While the factorization (\ref{fact0}) holds if we replace categories of perverse sheaves by derived categories, we can not expect to have (\ref{fact}) with derived categories instead of perverse sheaves because the kernel $(q,p)_!\Q$ is not $W$-equivariant on $[T/T]\times_{\car_T}[G/G]$ (see \cite[\S 5.4]{LL}).
\end{remark}

The construction of the left adjoint ${^*}\calI_{[T/N]}^G$ of $\calI_{[T/N]}^G$ is also clear from \cite[\S 2.9]{LL}.

\begin{theorem} (1) The co-unit 
$$
{^*}\calI_{[T/N]}^G\circ\calI_{[T/N]}^G\longrightarrow 1
$$
is an isomorphism.

\noindent (2) For any $K\in\calM([T/T])\simeq\calM(T)$ equipped with a $W$-equivariant structure,  

$$
\calI_{[T/N]}^G(\overline{K})=\Ind_{[T/T]}^G(K)^W
$$
where $\overline{K}$ satisfies $\pi_{[T/T]}^*(\overline{K})=K$.

\noindent (3) The functors $\calI_{[T/N]}^G$ and ${^*}\calI_{[T/N]}^G$ induces functors between categories of $F$-equivariant perverse sheaves and the following diagrams commute

$$
\xymatrix{\calM([T/N];F)\ar[rr]^{\calI_{[T/N]}^G}\ar[d]_\bfX&&\calM([G/G];F)\ar[d]^\bfX\\
\calC([T/N]^F)\ar[rr]^{\I_{[T/N]}^G}&&\calC([G/G]^F)} \hspace{2cm}  \xymatrix{\calM([G/G];F)\ar[rr]^{{^*}\calI_{[T/N]}^G}\ar[d]_\bfX&&\calM([T/N];F)\ar[d]^\bfX\\
\calC([G/G]^F)\ar[rr]^{{^*}\I_{[T/N]}^G}&&\calC([T/N]^F)}
$$
\label{Lind}\end{theorem}

\begin{proof}The first assertion is \cite[Theorem 7.8]{LL}, the second one follows from \cite[Remark 2.17]{LL}. 

Lusztig \cite{LusGreen} proved that for any smooth $\ell$-adic sheaf $K$ on $T$ equiped with a Weil structure $\varphi$, we have

\begin{equation}
{\bf X}_{\Ind_T^G(K,\varphi)}=R_T^G({\bf X}_{K,\varphi}).
\label{forI}\end{equation}
Moreover for any $(K,\varphi)\in\dcb(T;F)$ we have

$$
{\bf X}_{\Ind_T^G(K,\varphi)}(x)=\sum_{t\in T^F}{\bf X}_{K,\varphi}(t)\, N(t,x)
$$
where $N(\,,\,):T^F\times_{\car_T}[G/G]^F\rightarrow\Q$ is the characteristic function of $(q',p)_!\Q=\IC_{T\times_{\car_T}[G/G]}$.

Therefore, as the characteristic functions of $F$-equivariant smooth $\ell$-adic sheaves generate the space of all functions on $T^F$, the above formula (\ref{forI}) remains true for any $(K,\varphi)\in\dcb(T;F)$. 

Indeed, for $(K,\varphi)\in\dcb(T;F)$ write

$$
{\bf X}_{K,\varphi}=\sum_{i=1}^k\lambda_i X_{K_i,\varphi_i}
$$
where $\lambda_1,\dots,\lambda_k\in\Q$ and $(K_i,\varphi_i)$ are $F$-equivariant smooth $\ell$-adic sheaves on $T$.

Then

\begin{align*}
{\bf X}_{\Ind_T^G(K,\varphi)}(x)&=\sum_t{\bf X}_{K,\varphi}(t) N(t,x)\\
&=\sum_i\lambda_i\sum_t{\bf X}_{K_i,\varphi_i}(t) N(t,x)\\
&=\sum_i\lambda_i\,{\bf X}_{\Ind_T^G(K_i,\varphi_i)}(x)\\
&=\sum_i\lambda_i \,R_T^G({\bf X}_{K_i,\varphi_i})(x)\\
&=R_T^G\left(\sum_i\lambda_i\, {\bf X}_{K_i,\varphi_i}\right)(x)\\
&=R_T^G({\bf X}_{K,\varphi}).
\end{align*}
Recall (see Remark \ref{remdecomp}) that to $(\overline{K},\overline{\varphi})\in\calM([T/N];F)\simeq\calM([T/W];F)$ corresponds an object $(K,\theta,\varphi)\in\calM(T;W,F)$ and so  for each $w\in W$, we have an $F\circ w$-equivariant perverse sheaf $(K,\varphi_w)$ on $T$. Recall that $T_w$ is an $F$-stable maximal torus of $G$ such that the Frobenius $F$ on $T_w$  corresponds to the Frobenius $F\circ w$ on $T$. We regard $(K,\varphi_w)$ as an $F$-equivariant perverse sheaf $(K_w,\varphi_w)$ on $T_w$. By definition

$$
\I_{[T/N]}^G({\bf X}_{\overline{K},\overline{\varphi}})=\frac{1}{|W|}\sum_{w\in W}R_{T_w}^G({\bf X}_{K_w,\varphi_w}).
$$
On the other hand, let $\varphi^G$ be the $F$-equivariant structure on $\Ind_T^G(K)$ induced by $\varphi:F^*K\simeq K$ and let $\theta^G:W\rightarrow{\rm Aut}\left(\Ind_T^G(K)\right)$ be the induced action. By Formula (\ref{invariant}) we have

$$
{\bf X}_{\Ind_T^G(K,\varphi)^W}=\frac{1}{|W|}\sum_{w\in W}{\bf X}_{\Ind_T^G(K),\,\theta^G_w\circ\varphi^G}.
$$
Since
$$
{\bf X}_{\Ind_{T_w}^G(K_w,\varphi_w)}={\bf X}_{\Ind_T^G(K),\,\theta^G_w\circ\varphi^G}
$$
we get that

$$
{\bf X}_{\Ind_T^G(K,\varphi)^W}=\frac{1}{|W|}\sum_{w\in W}{\bf X}_{\Ind_{T_w}^G(K_w,\varphi_w)}
$$
and so we conclude from Formula (\ref{forI}) that

$$
\I_{[T/N]}^G({\bf X}_{\overline{K},\overline{\varphi}})={\bf X}_{\Ind_T^G(K,\varphi)^W}={\bf X}_{\calI_{[T/N]}^G(\overline{K},\overline{\varphi})}.
$$
The second equality being a consequence of the assertion (2).
\end{proof}

\subsection{Braverman-Kazhdan conjecture}\label{geointer1}

Let $T$ be a maximally split $F$-stable maximal torus of $G$ with normalizer $N$ and Weyl group $W=N/T$, $(G',F)$ a standard pair (see \S \ref{Fourierstandard}) and $\rho^\flat:G^\flat\rightarrow G'{^\flat}=G'$ a morphism that commutes with Frobenius. We let $L'$ be the centralizer of $\rho^\flat(T^\flat)$ in $G'$.

To simplify the presentation we will assume that $G'=\GL_n$, that $L'$ is of the form $(\GL_{n_1})^{a_1}\times\cdots\times (\GL_{n_r})^{a_r}$ with $n_1>\cdots>n_r$ (similar results will hold for arbitrary standard pairs $(G',F)$) and that $T'$ the maximal torus $\T_n$ of $\GL_n$ of diagonal matrices. We will use the following identifications (see \S \ref{reductive})

$$
N_{G'}(L')=L'\rtimes (S_{a_1}\times\cdots\times S_{a_r}), \hspace{.5cm} W_{G'}(L')=S_{a_1}\times\cdots\times S_{a_r}.
$$
The action of $S_{a_1}\times\cdots\times S_{a_r}$ on $L'$ preserves $T'$ and defines a natural embedding of $S_{a_1}\times\cdots\times S_{a_r}$ in $S_n=W_{G'}(T')$. The map $\rho^\flat$  induces thus a group homomorphism $W\rightarrow S_{a_1}\times\cdots\times S_{a_r}\hookrightarrow S_n$, $w\mapsto w'$ and so an action of $W$ on $T'$.
\bigskip

The morphism $\rho:T'\rightarrow T$ (obtained by duality from $\rho^\flat:T^\flat\rightarrow T'$) is $W$-equivariant.
\bigskip

Fix a non-trivial additive character $\psi$ of $\F_q$ and let $\mathcal{L}_\psi$ the Artin-Schreier sheaf on the affine line over $\overline{\F}_q$ equipped with its natural Weil structure $\varphi_\psi:F^*\mathcal{L}_\psi\simeq\mathcal{L}_\psi$ such that

$$
{\bf X}_{\mathcal{L}_\psi,\varphi_\psi^{(i)}}=\psi\circ\Tr_{\F_{q^i}/\F_q},
$$
for all positive integer $i$.

The trace $\Tr|_{T'}: T'\rightarrow\overline{\F}_q$ being $W$-invariant, the $F$-equivariant  perverse sheaf  $\Phi^{T'}:=(\Tr|_{T'})^*(\mathcal{L}_\psi,\varphi_\psi)[{\rm dim}\, T']$ on $T'$ gives a natural object $(\Phi^{T'},\theta^{T'},\varphi^{T'})$ of $\dcb(T';W,F)$.

\begin{proposition}The complex $\rho_!\Phi^{T'}$ is a perverse sheaf.
\end{proposition}

\begin{proof}As $\rho(T')$ is closed in $T$, we may assume without loss of generality that $\rho$ is surjective, i.e. that $\rho$ is the quotient map

$$
T'\rightarrow T'/S=T
$$
where $S:={\rm Ker}(\rho)$. Denote by $j:T'={\rm T}_n\hookrightarrow \mathbb{A}^n$ the natural open embedding.  The morphism $j$ is affine and quasi-finite and so $j_!$ preserves perverse sheaves.

 Let $m:\mathbb{A}^n\times\mathbb{A}^n\rightarrow\mathbb{A}^n$ be the multiplication (coordinate by coordinate) and consider the standard geometric Fourier transform

$$
\calF^{\mathbb{A}^n}:\calM(\mathbb{A}^n)\rightarrow\calM(\mathbb{A}^n),\hspace{1cm}K\mapsto \pr_{2\, !}\left(\pr_1^*(K)\otimes m^*(\Tr^*(\mathcal{L}_\psi))\right)[{\rm dim}\,T'].
$$
It preserves equivariance by $S$ (which acts by translation on $\mathbb{A}^n$) and so does its restriction 

$$
\calF^{T'}:=j^*\circ\calF^{\mathbb{A}^n}\circ j_!:\calM(T')\rightarrow\calM(T').
$$
 The operator $\calF^{T'}:\calM(T')\rightarrow\calM(T')$ descends to an operator $\calF^T:\calM(T)\rightarrow\calM(T)$ such that the diagram
 
 $$
\xymatrix{\calM(T')\ar[rr]^{\calF^{T'}}&&\calM(T')\\
\calM(T)\ar[rr]^{\calF^T}\ar[u]^{\rho^*}&&\calM(T)\ar[u]^{\rho^*}}
$$ 
commutes. Therefore

$$
\calF^T(\overline{\mathbb{Q}}_{\ell, 1})=\rho_!\Phi^{T'}
$$
where $\overline{\mathbb{Q}}_{\ell,1}$ is the constant sheaf supported at $1\in T$, is a perverse sheaf.

\end{proof}

We then get an object $(\Phi_\rho^{[T/N]},\varphi_\rho^{[T/N]})\in\calM([T/N];F)\simeq\calM([T/W];F)$ by descending the object $(\Phi^T,\theta^T,\varphi^T)$ of $\calM(T;W,F)$ obtained from 

$$
\rho_!(\Phi^{T'},\theta^{T'},\varphi^{T'})\in\calM(T;W,F)
$$
by multiplying its $W$-equivariant structure by the linear character $\alpha:w\mapsto (-1)^{\ell(w)+\ell(w')}$ of $W$ where $\ell(w)$ (resp. $\ell(w')$) denotes the length of $w$ in $W$ (resp. the length of $w'$ in $S_n$). 
\bigskip

\begin{lemma}We have

$$
{\bf X}_{\Phi^{[T/N]}_\rho,\varphi^{[T/N]}_\rho}=\epsilon_G\,  \phi^{[T/N]}_\rho.
$$
where $\phi^{[T/N]}_\rho$ is the function defined in \S \ref{Spectraldef}.
\label{lem1}\end{lemma}

\begin{proof}By Remark \ref{remdecomp}, we need to prove that

$$
{\bf X}_{\mathfrak{o}_w(\Phi^T,\theta^T,\varphi^T)}\in\calC(T^{F\circ w})
$$
corresponds to $\epsilon_G\, \phi^{T_w}_\rho\in\calC(T_w^F)$ under the identification $\calC(T_w^F)\cong\calC(T^{F\circ w})$.

\bigskip

We have the commutative diagram 

$$
\xymatrix{\dcb(T';W,F)\ar[rr]^{\rho_!}\ar[d]_{\mathfrak{o}_w}&&\dcb(T;W,F)\ar[d]^{\mathfrak{o}_w}\\
\mathcal{D}_c^b(T';F\circ w)\ar[rr]^{\rho_!}\ar[d]_\bfX&&\mathcal{D}_c^b(T;F\circ w)\ar[d]^\bfX\\
\mathcal{C}(T'{^{F\circ w}})\ar[rr]^{(\rho^{F\circ w})_!}\ar@{}[d]|*=0[@]{\cong}&&\mathcal{C}(T^{F\circ w})\ar@{}[d]|*=0[@]{\cong}\\
\calC({T'}_{w'}^F)\ar[rr]^{\rho_w^F}&&\calC(T_w^F)}
$$
from which we deduce that

\begin{align*}
{\bf X}_{\mathfrak{o}_w(\Phi^T,\theta^T,\varphi^T)}&=\alpha(w)\,\rho^{F\circ w}_!\left({\bf X}_{\mathfrak{o}_w(\Phi^{T'},\theta^{T'},\varphi^{T'})}\right)\\
&=(-1)^{{\rm dim}\, T'}\alpha(w)\,\rho_{w\, !}^F(\phi_o^{T'_{w'}})\\
&=(-1)^{{\rm dim}\, T'}\alpha(w)c_{T'_{w'},T_w}\phi^{T_w}_\rho\\
&=\epsilon_G\,\phi^{T_w}_\rho.
\end{align*}
\end{proof}

Put 

$$
(\Phi^G_\rho,\varphi_\rho^G):=\calI_{[T/N]}^G\left(\Phi^{[T/N]}_\rho,\varphi^{[T/N]}_\rho\right).
$$

\begin{theorem} 

\begin{equation}
{\bf X}_{\Phi_\rho^G,\varphi^G_\rho}=\epsilon_G\, \phi^G_\rho.
\label{mainconj}\end{equation}

\label{BKmainconj}\end{theorem}

\begin{proof}Follows from Theorem \ref{Lind}(3) and Formula (\ref{mainfor}).

\end{proof}

\begin{remark}Let $\sigma: T\rightarrow\GL_1$ be a character. We say that a cocharacter $\lambda:\GL_1\rightarrow T$ is \emph{$\sigma$-positive} (see \cite{NC}) if $\sigma\circ \lambda:\GL_1\rightarrow\GL_1$ is of the form $t\mapsto t^m$ for some positive integer $m$. The morphism $\rho^\flat:T^\flat\rightarrow T'=\T_n$ is given by $n$ characters $\lambda_1,\dots,\lambda_n$ which can be regarded as cocharacters of $T$ via $X(T^\flat)\simeq Y(T)$. Assume that the cocharacters $\lambda_1,\dots,\lambda_n$ are $\sigma$-positive. Braverman and Kazhdan \cite[Theorem 4.2(3)]{BK} proved that the complex $\mathcal{E}$ is an irreducible perverse sheaf on the image of $\rho:T'\rightarrow T$. Later on, Cheng and Ng\^o proved that this complex is actually a smooth $\ell$-adic sheaf on the image of $\rho$ \cite[Proposition 2.1]{NC}. Under the $\sigma$-positivity assumption and assuming that $\rho$ is surjective, we thus regard the complex ${\rm Ind}_T^G(\rho_!\Phi')$ as the intermediate extension of some semisimple smooth $\ell$-adic sheaf on the open subset of semisimple regular elements of $G$ on which the Weyl group $W$ acts.  Braverman and Kazhdan \cite{BK} conjectured that the characteristic function of ${\rm Ind}_T^G(\rho_!\Phi')^W$ coincides with the kernel $\phi^G_\rho$. In light of Theorem \ref{Lind}(2) we see that Theorem \ref{BKmainconj} proves a more general statement than their conjecture as we do not make any assumption of $\rho$. 
\label{remark}\end{remark}
\bigskip

\section{Appendix A}

The commutativity of Diagram (\ref{comdiagLie}) arises as a particular case of a result in \cite{letellier}. As the proof of the commutativity of Fourier transforms with Lusztig induction is relatively simple in the case of maximal tori (it is still a conjecture for arbitrary non-split Levi subgroups in arbitrary connected reductive groups, see \cite{letellier}), we give it for the convenience of the reader.
\bigskip

Assume that $G$ is an arbitrary connected reductive group with a geometric Frobenius $F:G\rightarrow G$. Its Lie algebra $\frakg$ is then equipped with a natural geometric Frobenius $F:\frakg\rightarrow\frakg$ and the adjoint action ${\rm Ad}$ of $G$ on $\frakg$ commutes with Frobenius. Let $T$ be a maximally split $F$-stable maximal torus with Lie algebra $\frakt$ and denote by $N$ the normalizer of $T$ in $G$ and put $W:=N/T$. We denote by $\car_\frakt$ the affine scheme $\frakt/\!/W$.

Define the geometric induction

$$
\calI_{[\frakt/N]}^\frakg:\calM([\frakt/N];F)\rightarrow\calM([\frakg/G];F)
$$
as in \S \ref{ind} with Lie algebras instead of groups. 

If the finite group $W$ acts on an $\overline{\F}_q$-scheme $X$, we denote by $\epsilon:\dcb(X;W)\rightarrow\dcb(X;W)$ the functor that maps an $W$-equivariant complex $(K,\theta)$ on $X$ to the $W$-equivariant complex $(K,\epsilon\theta)$ where $\epsilon\theta$ is the $W$-equivariant structure $\theta$ twisted by the sign character $\epsilon$ of $W$.

Define the geometric induction

$$
\calI_{[\frakt/N],\epsilon}^\frakg:\calM([\frakt/N];F)\rightarrow\calM([\frakg/G];F)
$$
where we use the kernel $\epsilon(\IC_{[\frakt/W]\times_{\car_\frakt}[\frakg/G]})$ instead of $\IC_{[\frakt/W]\times_{\car_\frakt}[\frakg/G]}$.
\bigskip

Assume given a $G$-invariant non-degenerate bilinear form $\langle\,,\,\rangle=\langle\,,\,\rangle_\frakg:\frakg\times\frakg\rightarrow \overline{\F}_q$ defined over $\F_q$ (invariant for the diagonal action of $G$ on $\frakg\times\frakg$). The existence of such bilinear form requires some restriction on the characteristic (see \cite[\S 2.5]{letellier}).
\bigskip

Consider the geometric Fourier transform $\calF^{[\frakg/G]}:\calM([\frakg/G];F)\rightarrow\calM([\frakg/G];F)$ defined by

$$
K\mapsto \pr_{2\, !}\left(\pr_1^*(K)\otimes\langle\,,\,\rangle^*(\calL_\psi)\right)[{\rm dim}\, \frakg]
$$
where $\pr_1,\pr_2:[(\frakg\times\frakg)/G]\rightarrow[\frakg/G]$ are the two projections ($G$ acting diagonally on $\frakg\times\frakg$).

Since the restriction of $\langle\,,\,\rangle$ to $\frakt\times\frakt$ remains non-degenerate and $N$-invariant, we also have a geometric Fourier transform $\calF^{[\frakt/N]}:\calM([\frakt/N];F)\rightarrow\calM([\frakt/N];F)$.
\bigskip

The aim of this section is to prove the following theorem.

\begin{theorem} The following diagram commutes

$$
\xymatrix{\calM([\frakt/N];F)\ar[d]_{\calF^{[\frakt/N]}}\ar[rr]^{\calI_{[\frakt/N],\epsilon}^\frakg\,(\nu_G)}&&\calM([\frakg/G];F)\ar[d]^{\calF^{[\frakg/G]}}\\
\calM([\frakt/N];F)\ar[rr]^{\calI_{[\frakt/N]}^\frakg}&&\calM([\frakg/G];F)}
$$
where $\nu_G$ is the dimension of the unipotent radical of a Borel subgroup of $G$.
\label{theoappendix}\end{theorem}

Fix an $F$-stable Borel subgroup $B$ containing $T$ with unipotent radical $U$. Let $\frakb$ and $\fraku$ be the Lie algebras of $B$ and $U$. 
\bigskip

Consider the variety

$$
Y:=\frakt\times\frakg\times (G/B)
$$
and the closed subschemes

$$
X:=\{(t,x,gB)\in Y\,|\, {\rm Ad}(g^{-1})(x)\in t+\fraku\},\hspace{1cm}X':=\{(t',x',gB)\in Y\,|\, g^{-1}x'g\in -t'+\fraku\}.
$$
Then the vector bundle $X'\rightarrow G/B$ is the orthogonal of $X\rightarrow G/B$ in $Y\rightarrow G/B$ with respect to the form $\langle\,,\,\rangle_\frakg+\langle\,,\,\rangle_\t$ on $\frakg\times\frakt$.
\bigskip

Denote by  $\calF^Y:\dcb(Y)\rightarrow\dcb(Y)$ the geometric Fourier transform relative to $G/B$ with kernel $\langle\,,\rangle_Y^*(\calL_\psi)$ where 

$$
\xymatrix{\langle\,,\,\rangle_Y: (\frakt\times\frakg)\times(\frakt\times\frakg)\times G/B\ar[rr]^-{\pr_{12}}&&(\frakt\times\frakg)\times(\frakt\times\frakg)\ar[rr]^-{\langle\,,\,\rangle_\frakg+\langle\,,\,\rangle_\frakt}&&\overline{\F}_q.}
$$
Namely if $\pr_1,\pr_2$ denote the two projections $(\frakt\times\frakg)\times(\frakt\times\frakg)\times G/B\rightarrow\frakg\times\frakt$, then 
$$
\calF^Y(K)=\pr_{2\, !}\left(\pr_1^*(K)\otimes \langle\,,\,\rangle_Y^*(\calL_\psi)\right)[{\rm dim}\, \frakt\times\frakg].
$$
The following result is straightforward.

\begin{lemma}
$$
\calF^Y(\overline{\mathbb{Q}}_{\ell, X})\simeq\overline{\mathbb{Q}}_{\ell,X'}(-{\rm dim}\, \frakb).
$$
\end{lemma}
Moreover if $p:Y\rightarrow\frakt\times\frakg$ denotes the projection, then 

$$
p_!\circ \calF^Y=\calF^{\frakt\times\frakg}\circ p_!
$$
and so

$$
\calF^{\frakt\times\frakg}(p_!\overline{\mathbb{Q}}_{\ell,X})\simeq p_!\overline{\mathbb{Q}}_{\ell,X'}(-{\rm dim}\,\frakb).
$$
Put $S:=\frakt\times_{\car_\frakt}\frakg$ and let $S'\subset\frakt\times\frakg$ be the subscheme of pairs $(t',x')$ such that the semisimple part of $x'$ is ${\rm Ad}(G)$-conjugate to $-t'$. 

The projections $\pi:X\rightarrow S$ and $\pi':X'\rightarrow S'$ being small resolutions of singularities we have

$$
p_!\overline{\mathbb{Q}}_{\ell,X}=\pi_!\overline{\mathbb{Q}}_{\ell,X}=\IC_S,\hspace{1cm}p_!\overline{\mathbb{Q}}_{\ell,X'}=\pi'_!\overline{\mathbb{Q}}_{\ell,X'}=\IC_{S'}.
$$
Therefore 

\begin{equation}
\calF^{\frakt\times\frakg}(\IC_S)\simeq\IC_{S'}(-{\rm dim}\, \frakb).
\label{iso}\end{equation}
As the chosen Borel $B$ is $F$-stable, this isomorphism is compatible with $F$-equivariant structures. However, this isomorphism depends on the choice of $B$ and so \emph{apriori} does not preserves the obvious $W$-equivariant structures on both sides (where the action of $W$ on $\frakt\times\frakg$ is given by the action of $W$ on the first factor).
\bigskip

\begin{proposition}The isomorphism (\ref{iso}) induces an isomorphism

\begin{equation}
\calF^{\frakt\times\frakg}(\epsilon(\IC_S))\simeq \IC_{S'}(-{\rm dim}\, \frakb).
\label{iso1}\end{equation}
in $\calM(S';W,F)$.

\label{isoprop}\end{proposition}

To prove the proposition we need to see that the obvious $W$-equivariant structure on $\IC_{S'}$ and the one induced by that of $\calF^{\frakg}(\IC_S)$ through the isomorphism (\ref{iso}) differs by the sign character of $W$. 

Notice that since $\IC_{S'}$ is (up to a shift) a simple perverse sheaf, any two $W$-equivariant structures on $\IC_{S'}$  differs by a linear character of $W$. We compute this linear character on the global compactly supported cohomology.

\begin{lemma} The action of $W$ on the top cohomology of $R\Gamma_c(S',\IC_{S'})$ is trivial.
\end{lemma}

\begin{proof}Notice that the top cohomology of $R\Gamma_c(S',\IC_{S'})$ is the top compactly supported cohomology of $X'$ and it is certainly well-know that the Springer action of $W$ on $H_c^{\rm top}(X',\Q)$ is trivial. As we could not locate a reference in the literature, we give a proof involving $S'$ as an essential ingredient.

We have

$$
\calH^{\rm top}R\Gamma_c(S',\IC_{S'})=H^{2d}_c(S',\IC_{S'})
$$
where $d={\rm dim}\, S'={\rm dim}\,\frakg$.

Denote by $S'_{\rm rss}$ the open subset of $S'$ of semisimple regular elements and by $j$ the inclusion $S'_{\rm rss}\hookrightarrow S'$. The map

\begin{equation}
H_c^{2d}(S'_{\rm rss},\Q)\rightarrow H^{2d}_c(S',\IC_{S'})
\label{isoco}\end{equation}
induced by the morphism $j_!j^*\IC_{S'}\rightarrow\IC_{S'}$ is naturally $W$-equivariant. 

The action of $W$ on $H^{2d}_c(S'_{\rm rss},\Q)$ is induced by the action of $W$ on the set of irreducible components of $S'_{\rm rss}$ of maximal dimension. As $S'_{\rm rss}$ is irreducible, this action is thus trivial.

We are thus reduced to prove that the morphism (\ref{isoco}) is an isomorphism. To see that we use the small map $\pi':X'\rightarrow S'$. We have the following commutative diagram

$$
\xymatrix{H_c^{2d}(S'_{\rm rss},\Q)\ar[rr]\ar[d]_{\simeq}&&H_c^{2d}(S',\IC_{S'})\ar[d]^{\simeq}\\
H_c^{2d}(X'_{\rm rss},\Q)\ar[rr]&&H_c^{2d}(X',\Q)}
$$
The bottom arrow is an isomorphism as 

$$
{\rm dim}\, (X'\backslash X'_{\rm rss})<d.
$$
\end{proof}

A simple calculation shows that

$$
R\Gamma_c(\frakt\times\frakg,\calF^{\frakt\times\frakg}(\IC_S))=\IC_{S,0}[-{\rm dim}\,\frakt\times\frakg](-{\rm dim}\,\frakt\times\frakg)\\
$$
where $\IC_{S,0}$ denotes the pullback of $\IC_S$ at $0\in S$.

Proposition \ref{isoprop} follows thus from the following lemma.

\begin{lemma}The group $W$ acts on $\calH^{\rm top}(\IC_{S,0}[-{\rm dim}\, \frakg\times\frakt])=\calH^{2{\rm dim}\,(G/B)}\IC_{S,0}$ by the sign character.
\end{lemma}

\begin{proof}If we notice that $\calH^{2{\rm dim}\,G/B}(\IC_{S,0})=H^{2{\rm dim}\, G/B}(G/B,\Q)$ then the above lemma is well-known as the Springer action of $W$ on the top cohomology of $G/B$ is known to be  the sign character. This can be indeed computed using the result of Borho-MacPherson \cite{BM} saying that the Springer action of $W$ on $H^i(G/B,\Q)$ coincides with the so-called classical action of $W$ on $H^i(G/B,\Q)=H^i(G/T;\Q)$.
\end{proof}

\begin{proof}[Proof of Theorem \ref{theoappendix}]Since $\calM([\frakt/N])\simeq\calM([\frakt/W])$, it is sufficient to prove the analogous statement with $[\frakt/W]$ instead of $[\frakt/N]$. Consider the quotient stacks

$$
\calS:=[S/G],\hspace{1cm}\calS':=[S'/G]
$$
for the action of $G$ on $\frakg$. Then $\calS$ and $\calS'$ are left with an action of $W$ on the other factor.
\bigskip

The isomorphism (\ref{iso1}) descends to an isomorphism

$$
\calF^{[\frakt/W]\times[\frakg/G]}\left(\epsilon(\IC_{[\calS/W]})(\nu_G)\right)\simeq \IC_{[\calS'/W]}(-{\rm dim}\, \frakt).
$$
Since we have an isomorphism of functors (involutivity of Fourier)

$$
\calF^{[\frakt/W]}\circ\calF^{[\frakt/W]}\simeq a^* \, ({\rm dim}\, \frakt)
$$
where $a:[\frakt/W]\rightarrow[\frakt/W]$ is induced by the map $\frakt\rightarrow\frakt$, $t\mapsto -t$, we deduce that

\begin{equation}
(1\times\calF^{[\frakg/G]})\left(\epsilon(\IC_{[\calS/W]})(\nu_G)\right)\simeq (\calF^{[\frakt/W]}\times 1)(\IC_{[\calS/W]}),
\label{Fourier-iso}\end{equation}
where $1\times \calF^{[\frakg/G]}:\calM([\frakt/W]\times[\frakg/G])\rightarrow\calM([\frakt/W]\times[\frakg/G])$ is obtained by doing Fourier transform on the first factor and nothing on the second, i.e. it is defined from the cohomological correspondence

$$
\xymatrix{[\frakt/W]\times[\frakg/G]&&[\frakt/W]\times[(\frakg\times\frakg)/G]\ar[rr]^-{{\rm Id}_{[\frakt/W]}\times\pr_2}\ar[ll]_-{{\rm Id}_{[\frakt/W]}\times\pr_1}&&[\frakt/W]\times[\frakg/G]}
$$
with kernel $(\langle\,,\,\rangle_\frakg^*\calL_\psi)\boxtimes \overline{\mathbb{Q}}_{\ell,[\frakt/W]}$, and where $\calF^{[\frakt/W]}\times 1:\calM([\frakt/W]\times[\frakg/G])\rightarrow\calM([\frakt/W]\times[\frakg/G])$ is defined by doing  Fourier transform on the second factor and nothing on the first one.
\bigskip

A simple calculation shows that the composition of functors $\calF^{[\frakg/G]}\circ\calI_{[\frakt/W],\epsilon}^\frakg\,(\nu_G):\calM([\frakt/W])\rightarrow\calM([\frakg/G])$ is given by the cohomological correspondence

\begin{equation}
\xymatrix{[\frakt/W]&&[\frakt/W]\times[\frakg/G]\ar[rr]\ar[ll]&&[\frakg/G]}
\label{cor2}\end{equation}
whose kernel is the left hand side of (\ref{Fourier-iso}), while the composition of functors $\calI_{[\frakt/W]}^\frakg\circ\calF^{[\frakt/W]}$ is given by the correspondence (\ref{cor2}) with kernel the right hand side of (\ref{Fourier-iso}).

\end{proof}

\section{Appendix B}

For a subvariety of $\gl_n$, we denote by $\Tr_Z$ and $\det_Z$ the restriction of the trace and the determinant to $Z$. Denote by $\B_n=\T_n\U_n$ the Borel subgroup of $\GL_n$ of upper triangular matrices, by ${\rm N}_n$ the normalizer of $\T_n$ in $\GL_n$ and ${\rm W}_n={\rm N}_n/\T_n$.

\begin{theorem}We have

$$
({\det}_{\GL_n})_!\left(\Tr_{\GL_n}^*\calL_\psi\right)[{\rm dim}\, \GL_n]({\rm dim}\, \U_n)\simeq({\det}_{\T_n})_!\left(\Tr_{\T_n}^*\calL_\psi\right)[{\rm dim}\, \T_n].
$$
\end{theorem}

\begin{proof}Let ${\rm X}_n$ be the complementary of $\B_n$ in $\GL_n$. Then by the Bruhat decomposition

$$
{\rm X}_n=\coprod_{w\in {\rm W}_n\backslash\{1\}}\B_n\dot{w}\U_{n,w},
$$
where $\dot{w}$ denotes a representative of $w$ in ${\rm N}_n$ and $\U_{n,w}:=\U_n\cap \dot{w}^{-1}\U_n\dot{w}$. 

We have a distinguished triangle 

$$
\xymatrix{({\det}_{{\rm X}_n})_!\left(\Tr_{{\rm X}_n}^*\calL_\psi\right)\ar[r]&({\det}_{\GL_n})_!\left(\Tr_{\GL_n}^*\calL_\psi\right)\ar[r]&({\det}_{\B_n})_!\left(\Tr_{\B_n}^*\calL_\psi\right)\ar[r]&}
$$
Since ${\det}_{\B_n}$ factorizes through ${\det}_{\T_n}$ via the projection $\B_n\rightarrow\T_n$ and since $\Tr_{\B_n}^*\calL_\psi$ is the pullback of $\Tr_{\T_n}^*\calL_\psi$ along that projection, we have

$$
({\det}_{\B_n})_!\left(\Tr_{\B_n}^*\calL_\psi\right)\simeq({\rm det}_{\T_n})_!\left(\Tr_{\T_n}^*\calL_\psi\right)[-2{\rm dim}\, \U_n](-{\rm dim}\, \U_n).
$$
It remains to see that $({\det}_{{\rm X}_n})_!\left(\Tr_{{\rm X}_n}^*\calL_\psi\right)=0$. For $w\neq 1$, put

$$
X_{w,n}:=\B_n\dot{w}\U_{w,n}.
$$
It is enough to show that 

$$
({\det}_{X_{w,n}})_!\left(\Tr_{X_{w,n}}^*\calL_\psi\right)=0,
$$
for all $w\neq 1$. The morphism ${\det}_{X_{w,n}}$ factorizes through the projection $p_w:X_{w,n}\rightarrow \T_n\dot{w}\U_{w,n}$ which is a $\U_n$-torsor. It is thus enough to prove that

$$
(p_w)_!\left(\Tr_{X_{w,n}}^*\calL_\psi\right)=0.
$$
Let $x\in \T_n\dot{w}\U_{w,n}$ and $p_{w,x}:\U_nx\rightarrow \{x\}$. We need to see that

\begin{equation}
(p_{w,x})_!\left(\Tr_{\U_nx}^*\calL_\psi\right)=0.
\label{intereq}\end{equation}

Writing $\U_n=1+\fraku_n$, we see that
$$
\Tr_{\U_nx}^*\calL_\psi\simeq \Tr_x^*\calL_\psi\boxtimes \Tr_{\fraku_nx}^*\calL_\psi.
$$
Consider 

$$
f:\fraku_nx\simeq\fraku_n\rightarrow \bigoplus_{i<j}\mathbb{A}^1,\,\,\{u_{ij}\}_{i<j}\mapsto\sum_{i<j}u_{ij}x_{ji}.
$$
Then $\Tr_{\fraku_nx}^*\calL_\psi=f^*\left(\bigboxtimes_{i<j}\mathcal{L}_\psi\right)=\bigboxtimes_{i<j}\mathcal{L}_{\psi,x_{ji}}$ where $\mathcal{L}_{\psi,x_{ji}}$ is the pullback of $\mathcal{L}_\psi$ along the map $\mathbb{A}^1\rightarrow\mathbb{A}^1$, $u\mapsto ux_{ji}$. 

As $x\notin\B_n$, for some $i<j$ we have $x_{ji}\neq 0$ and so $\mathcal{L}_{\psi,x_{ji}}\neq \Q$. Therefore, the proper pushforward of $\mathcal{L}_{\psi, x_{ji}}$ on a point is zero. From Kunn\"eth formula we deduce (\ref{intereq}).
\end{proof}

\end{document}